\documentclass[english,11pt, oneside, a4paper]{amsart}

\usepackage{amsfonts,amsmath,amsthm,amssymb,amscd,latexsym,enumerate,etoolbox,mathrsfs,comment, graphicx}
\usepackage[all]{xy}
\usepackage{tikz-cd}
\usepackage{float}
\usepackage[left=2.5cm,right=2.5cm,top=2cm,bottom=2.5cm,bindingoffset=5mm]{geometry}
\usepackage[T1]{fontenc} 
\usepackage[utf8]{inputenc}
\usepackage{mathtools}
\usepackage{hyperref}
\usepackage{theoremref}
\usepackage{quiver}
\usepackage{tikz}
\usepackage{standalone} 
\usepackage{tikzscale} 
\usepackage{enumitem} 

\usepackage[normalem]{ulem} 

\usetikzlibrary{backgrounds,fit,positioning}

\usepackage[english]{babel}

\definecolor{myblue}{RGB}{0, 0, 255}
\definecolor{myred}{RGB}{255, 0, 25}
\newtheorem{theorem}{Theorem}[section]
\newtheorem{proposition}[theorem]{Proposition}
\newtheorem{lemma}[theorem]{Lemma}
\newtheorem{corollary}[theorem]{Corollary}
\theoremstyle{definition}
\newtheorem{remark}[theorem]{Remark}
\newtheorem{definition}[theorem]{Definition}
\newtheorem{example}[theorem]{Example}

\newcommand{\mb}[1]{\mathbb{#1}}

\newcommand{\norm}[1]{\left\| #1  \right\|}

\newcommand{\Z}{\mb Z}
\newcommand{\N}{\mb N}
\newcommand{\C}{\mb C}
\newcommand{\V}{\mathcal{V}}

\newcommand{\E}{\mathcal{E}}
\newcommand{\CP}{\mathcal{O}}
\newcommand{\calL}{\mathcal{L}}

\newcommand{\K}{\mathcal{K}}
\newcommand{\F}{\mathcal{F}}

\DeclareMathOperator{\Span}{span}

\DeclareMathOperator{\rank}{rank}

\DeclareMathOperator{\orb}{Orb}
\DeclareMathOperator{\interior}{int}

\DeclareMathOperator{\id}{id}

\title[Cuntz--Pimsner algebras of partial automorphisms twisted by vector bundles I]{Cuntz--Pimsner algebras of partial automorphisms twisted by vector bundles I: Fixed point algebra, simplicity and the tracial state space}
\author{Aaron Kettner}

\address{Department of Abstract Analysis\\ Institute of Mathematics, Czech Academy of Sciences, \v{Z}itn\'a 25, 115 67 Prague 1, Czech Republic}
\email{kettner@math.cas.cz}

\date{\today}
\subjclass[2020]{37A55, 46L35, 46L08}

\thanks{Funded by GA\v{C}R project GF22-07833K and \mbox{RVO: 67985840}. Part of this work was carried out while funded by GA\v{C}R project 20-17488Y. This work has been supported by Charles University Research Centre
program No. UNCE/24/SCI/022, and by the Charles University project SVV-2023-260721. The author is currently funded by GA\v{C}R project G25-15403K.}

\keywords{Partial automorphisms, $\mathrm{C}^*$-correspondences, classification of nuclear \mbox{$\mathrm{C}^{*}$-algebras}}

\begin{document}

\begin{abstract}
We associate a $C^*$-algebra to a partial action of the integers acting on the base space of a vector bundle, using the framework of Cuntz--Pimsner algebras. We investigate the structure of the fixed point algebra under the canonical gauge action, and show that it arises from a continuous field of $C^*$-algebras over the base space, generalising results of Vasselli. We also analyse the ideal structure, and show that for a free action, ideals correspond to open invariant subspaces of the base space. This shows that if the action is free and minimal, then the Cuntz--Pimsner algebra is simple. In the case of a line bundle, we establish a bijective corrrespondence between tracial states on the algebra and invariant measures on the base space. This generalizes results about the $C^*$-algebras associated to homeomorphisms twisted by vector bundles of Adamo, Archey, Forough, Georgescu, Jeong, Strung and Viola. 
\end{abstract}

\maketitle

\tableofcontents

\section{Introduction}

The machinery of Cuntz--Pimsner algebras, as pioneered in \cite{Pimsner:1997} and generalized in \cite{Katsura:2004}, provides a way to construct a $C^*$-algebra from a $C^*$-correspondence. Given two $C^*$-algebras $A$ and $B$, a $C^*$-correspondence is an $A$-$B$-bimodule with some extra structure. More precisely, if $A$ and $B$ are $C^*$-algebras, then an $A$-$B$ $C^*$-correspondence consists of a right Hilbert $B$-module together with a left $A$-action. They can be viewed as a class of morphisms of $C^*$-algebras more general than the usual $\ast$-homomorphisms. In that spirit, an $A$-$A$ $C^*$-correspondence---or $C^*$-correspondence over $A$ for short---is a natural generalization of a *-endomorphism  of $A$. 

Crossed products, especially crossed products by the integers, have been one of the most well-studied objects in $C^*$-algebra theory. For example, they appear in the construction of the noncommutative torus \cite{Rieffel:1981}, which is one of the most foundational examples in noncommutative geometry, and have been used to classify Cantor minimal systems \cite{GioPutSkau:orbit}. They also play an important role in the Elliott classification program. Indeed, classifiability of crossed products has been studied extensively, see for example \cite{Hirshberg2017}. Viewing a $C^*$-correspondence as a generalized morphism, Cuntz--Pimsner algebras are a natural generalization of crossed products by the integers. The resulting $C^*$-algebras are both tractable and include a wide range of naturally occurring examples. 

Of particular interest is the case when the coefficient $C^*$-algebra of the $C^*$-correspondence is commutative. In addition to generalizing crossed products by homeomorphisms, these Cuntz--Pimsner algebras include the well-studied graph $C^*$-algebras, and even the topological graph $C^*$-algebras from \cite{katsura:2004_2}. They also include crossed products by partial actions of the integers on commutative $C^*$-algebras, in the sense of \cite{Exel:2017}. 

The study of partial actions in the context of $C^*$-algebras was initiated by Exel in \cite{Exel:1994}, where he considered partial actions by the integers. Any crossed product by $\mathbb{Z}$ admits an action of $\mathbb{T}$ (the dual action). The converse need not hold, that is, a $C^*$-algebra with an action of $\mathbb{T}$ need not be isomorphic to a crossed product, even if the action is nice enough. However, in \cite{Exel:1994}, Exel showed that if the action is regular \cite[Definition 4.4]{Exel:1994}, then it is indeed a crossed product by a partial integer action. This generalised a result by Landstad for global actions \cite{Landstad:1979}. Partial actions have since been extended to more general discrete groups and semigroups and have become an important area of study in their own right \cite{Exel:2017}. In particular, they have turned out to be an important source of examples in Elliott's classification program \cite{deeleyputnamstrung:2018jiangsu, deeleyputnamstrung:2024}.  

In \cite{Abadie2003}, Abadie, Eilers, and Exel also considered the question of the structure of circle actions on $C^*$-algebras. They defined crossed products by Hilbert bimodules, and showed that a $C^*$-algebra is isomorphic to such a crossed product if and only if it admits a circle action which is semi-saturated \cite[Definition 4.1]{Exel:1994}. They give an example of an algebra admitting such a circle action, whose fixed point subalgebra is a unital commutative $C^*$-algebra, but which is not a crossed product by a partial automorphism \cite[Example 3.3.]{Abadie2003}.

Crossed products by partial automorphisms of compact Hausdorff spaces have a particularly nice theory. For example, there is a bijective correspondence between ideals of the crossed product and open subsets of the space invariant under the action \cite{Exel:2017}. This should be compared to the significantly more complicated theory for graph $C^*$-algebras (\cite{AnHuefRaeburn:1997, Katsura:2006}). Furthermore, the fixed point algebra of a partial crossed product under the canonical circle action recovers $C(X)$, the algebra of continuous functions on the space. If the action is topologically free then this fixed point algebra is a Cartan subalgebra of the crossed product \cite{Renault:2008}. Cartan subalgebras have received significant attention in recent years, thanks to the role they play in the classification program and the question of whether or not every nuclear $C^*$-algebra satisfies the Universal Coefficient Theorem (UCT) \cite{BarlakLi2017, Li:Cartan}.

Crossed products by minimal partial automorphisms on compact metric spaces with finite covering dimension also have finite nuclear dimension \cite{Geffen2021}, meaning they are covered by the classification theorem that says simple, separable, nuclear, unital, infinite-dimensional $C^*$-algebras which satisfy the UCT and have finite nuclear dimension are classified up to isomorphism by the Elliott invariant (see for example \cite{WinterICM}).

A important part of the Elliott invariant is the tracial state space. For crossed products by free partial automorphisms, there is a bijective correspondence between tracial states of the crossed product by a homeomorphism and invariant measures on the space. The tracial state space is a crucial piece of the Elliott invariant. Indeed, in \cite{Elliott1993:AI} Elliott was able to extend classification outside the case of real rank zero to simple inductive limits of interval algebras using $K$-theory along with the tracial state space. (In the case of real rank zero, tracial information is contained in the ordered $K_0$-group.) Thanks to classification, we see that isomorphism in certain classes of $C^*$-algebras---for example the crossed products of $S^d$, $d \geq 3$ and odd integer, by minimal diffeomorphisms, or the simple inductive limits of dimension drop algebras with vanishing $K_1$-group---depend only on their tracial state space \cite{Str:XxSn, JiangSu99}.

In summary, there is an abundance of nice structural results for partial crossed products by the integers over commutative $C^*$-algebras. Since Cuntz--Pimsner algebras arising from $C^*$-correspondences are a natural generalisation of such crossed products, this raises the following question: If $X$ is a locally compact Hausdorff space, then what is the largest class of $C^*$-correspondences over $C(X)$ for which the techniques used for crossed products by the integers can be adapted? 

In \cite{adamo2023} the authors studied $C^*$-correspondences whose left action is implemented by a homeomorphism. More precisely, they considered a compact metric space $X$, a homeomorphism $\alpha$ and a vector bundle $\V$ over $X$. The space of continuous sections $\Gamma(\V)$ of the vector bundle is a right Hilbert $C(X)$-module. In fact, by the Serre--Swan theorem \cite{Swan:1962} every Hilbert $C(X)$-module is of that form as long as it is finitely generated and projective. The left action is defined as 
\begin{equation}\label{eq:left action}
    \varphi(f)\xi=\xi(f\circ\alpha)
\end{equation} 
for $f\in C(X)$ and $\xi\in\Gamma(\V)$. If $\V$ is the trivial line bundle $X\times\C$, then $\Gamma(\V)$ is isomorphic to $C(X)$ as right Hilbert $C(X)$-modules. The associated Cuntz--Pimsner algebra is isomorphic to the crossed product $C(X)\ltimes_\alpha\Z$, see \cite[Proposition 2.23]{Muhly:1998}. 

Thus, the class of $C^*$-correspondences studied in \cite{adamo2023} naturally arises from the desire to generalize techniques used for crossed products by the integers on commutative $C^*$-algebras. This class consists of precisely those $C^*$-correspondences over $C(X)$ whose left action is of the same form as for crossed products, with the additional assumption of finitely generatedness and projectiveness on the Hilbert module. This last assumption is of a technical nature, and can likely be removed.

In this paper we generalize some of the results of \cite{adamo2023} to partial $\Z$-actions on locally compact spaces. The $C^*$-correspondences we consider are those with a left action of the form (\ref{eq:left action}). However, we replace the homeomorphism $\alpha$ with a homeomorphism $\theta$ between two open subsets $U$ and $V$ of $X$. Such a partial homeomorphism on $X$ gives rise to a partial $\Z$-action in the sense of \cite{Exel:2017}. We usually call $\theta$ a \emph{partial automorphism}.

We have noted above that for partial crossed products, the coefficient algebra $C(X)$ can be recovered as the fixed point algebra of the canonical gauge action, and that it is a Cartan subalgebra if the action is topologically free. For partial automorphisms twisted by vector bundles, this is only true if the vector bundle is a line bundle. If the vector bundle is higher rank, then one should expect the fixed point algebra to be noncommutative and closely related to the UHF algebra $M_{n^\infty}$, where $n$ is the rank of the vector bundle. This is the case since the fixed point algebra of the Cuntz algebra $\CP_n$ under the canonical gauge action is given by $M_{n^\infty}$ \cite{Cuntz:1977}, and $\CP_n$ can be obtained in our formalism from considering the trivial action and vector bundle on a one-point space. We investigate the structure of the fixed point algebra in our setting in Section \ref{sect:cuntz pimsner algebras from twisted partial autos}.

As expected, for a free action the ideal structure of the Cuntz--Pimsner algebra is determined by dynamical properties of the action in the same way that this is the case for partial crossed products by the integers. For the tracial state space we generalise results from \cite{adamo2023}, again in the case of a free action: If the vector bundle is a line bundle, meaning that all its fibers are one-dimensional, then the usual bijective correspondence between traces and invariant measures holds. In particular, those Cuntz--Pimsner algebras are stably finite.

Further connections to classification will be explored in a second paper: Under suitable conditions on the compact Hausdorff space $X$, it follows from the general theory of Cuntz--Pimsner algebras \cite{Katsura:2004} that the Cuntz--Pimsner algebra of a partial automorphism twisted by a vector bundle is nuclear, separable and satisfies the UCT. In this second paper we will investigate under which assumptions our Cuntz--Pimsner algebras have finite nuclear dimension (and are thus classifiable), and again find that it is analogous to the case of partial crossed products. In future work, we will determine the Elliott-invariant of our algebras. Additionally to the traces, this includes $K$-theory as well as the pairing between $K$-theory and traces.
\subsection{Summary of the paper}
Section \ref{sect:hilbert modules correspondences and cuntz pimsner algebras} recalls the necessary background on Hilbert modules, $C^*$-correspondences and Cuntz--Pimsner algebras. It contains very little that is new, apart from some basic results which we could not find in the literature. 

The main results of the paper are contained in Sections \ref{sect:cuntz pimsner algebras from twisted partial autos} and \ref{sect:ideal structure and tracial state space}. We start in Section \ref{sect:cuntz pimsner algebras from twisted partial autos} by discussing Hilbert $C_0(X)$-modules. Since we want to consider locally compact spaces, we need a generalization of the Serre--Swan theorem from the compact case. For this reason we introduce properties of Hilbert $C_0(X)$-modules which we call $\sigma$-finitely generated and $\sigma$-projective. By definition, a Hilbert $C_0(X)$-module has these properties whenever any restriction to a compact subset of $X$ is finitely generated and projective, respectively. We outline how to show that such modules are exactly the Hilbert $C_0(X)$-modules $\Gamma_0(\V)$ of continuous sections vanishing at infinity of a vector bundle $\V$ over $X$. This is the desired Serre--Swan type theorem. 

After reviewing the basics of partial actions, we construct a $C^*$-correspondence from a vector bundle and a partial automorphism acting on the base space. The first main result of the paper concerns the fixed point algebra under the canonical gauge action. We show that it arises from a continuous field of $C^*$-algebras. The fiber over a point $x\in X$ is either a matrix algebra or a UHF algebra, depending on how often the partial automorphism can be applied to $x$. This generalizes results from \cite{Vasselli:2005}. 

Section \ref{sect:ideal structure and tracial state space} investigates the ideal structure as well as the tracial state space of our Cuntz--Pimsner algebras. Regarding the ideal structure, we show that the known results about (partial) crossed products still hold in our setting. Under the additional assumption of freeness, ideals in the Cuntz--Pimsner algebra correspond to open invariant subsets of the base space $X$, see Theorem \ref{thm:ideals in Cuntz--Pimsner algebra}. Thus the Cuntz--Pimsner algebra associated to a free and minimal partial automorphism is simple, regardless of what the vector bundle looks like. We show that if the vector bundle is rank one, then the tracial state space is affinely homeomorphic to the compact convex set of invariant measures on $X$, see Proposition \ref{prop:affine homeomorphism}. This aligns well with results from \cite{adamo2023}.\\

\textit{Acknowledgements:} I would like to thank my supervisor Karen Strung for giving me this project, and for her relentless encouragement and support. This work will form part of the author's PhD thesis.
\section{Hilbert modules, \texorpdfstring{$C^*$}{C*}-correspondences and Cuntz--Pimsner algebras}\label{sect:hilbert modules correspondences and cuntz pimsner algebras}
In this paper we study the Cuntz--Pimsner algebra associated to a $C^*$-correspondence over a commutative $C^*$-algebra $C_0(X)$, which we construct from a partial automorphism on the locally compact Hausdorff space $X$. In this section we introduce the necessary background:  We start by reviewing Hilbert $C^*$-modules, since $C^*$-correspondences are Hilbert $C^*$-modules with extra structure. We then turn to $C^*$-correspondences, their morphisms and representations, and finally to Cuntz--Pimsner algebras.

\subsection{Hilbert \texorpdfstring{$C^*$}{C*}-modules}\label{sect:Hilbert modules}

The definition as well as basic properties of Hilbert $A$-modules, where $A$ is a $C^*$-algebra, can be found in \cite[Chapter 1]{Lance:1995} and \cite[Chapter 15.2]{Wegge-Olsen:1993}. We will only consider right Hilbert $A$-modules. They will usually be denoted by $\E$, and the $A$-valued inner product by $\langle\cdot,\cdot\rangle_\E$. We write $\langle\E,\E\rangle_\E\coloneqq\overline{\Span}\{\langle\xi,\eta\rangle_\E:\xi,\eta\in\E\}$, where $\overline{\Span}(S)$ denotes the closure of the linear span of a subset $S$. Then $\langle\E,\E\rangle_\E$ is an ideal in $A$, and hence $\E\langle\E,\E\rangle_\E$ is contained in $\E$. It is an easy application of Cohen's factorization theorem that they are in fact equal. A right Hilbert $A$-module $\E$ is called \emph{full} if $\langle\E,\E\rangle_\E=A$. 

The $C^*$-algebra of adjointable $A$-linear operators on $\E$ is denoted by $\calL(\E)$. For $\xi,\eta\in\E$ we define an operator $\theta_{\xi,\eta}\in\calL(\E)$ by $\theta_{\xi,\eta}\zeta=\xi\langle\eta,\zeta\rangle_\E$. We call $\theta_{\xi,\eta}$ a \emph{rank-one operator}. The subalgebra of compact operators $\K(\E)\subset\calL(\E)$ is defined as the closure of the $A$-linear span of rank-one operators inside of $\calL(\E)$. The $C^*$-algebra $\calL(\E)$ is isomorphic to the multiplier algebra of the compacts, $M(\K(\E))$.

\begin{definition}\label{def:morphism of Hilbert modules}
    Let $A$ and $B$ be two $C^*$-algebras, and let $\E$ and $\F$ be right Hilbert $A$- and $B$-modules, respectively. A \emph{morphism of Hilbert modules} from $\E$ to $\F$ is a pair $(\phi,T)$ consisting of a $\ast$-homomorphism $\phi$ from $A$ to $B$ and a linear map $T$ from $\E$ to $\F$ satisfying
    \begin{enumerate}
        \item $\langle T(\xi),T(\eta)\rangle_\F=\phi(\langle\xi,\eta\rangle_\E)$ for $\xi,\eta\in \E$;
        \item $T(\xi a)=T(\xi)\phi(a)$ for $a\in A,\xi\in\E$. 
    \end{enumerate}
    If $A=B$ and the $\ast$-homomorphism $\phi$ is not made explicit, then by a right Hilbert $A$-module morphism $T$ from $\E$ to $\F$ we mean the pair $(\id_A,T)$.
\end{definition}

Let $A$ be a $C^*$-algebra and let $\E$ be a right Hilbert $A$-module. Take a subset $S$ of $\E$ and write 
\[\Span_A S=\{\sum_{i=1}^n s_ia_i|n\in\N, s_1,...,s_n\in S, a_1,...,a_n\in A\}\]
for the $A$-linear span of $S$.
     We say that $S$ \emph{generates} $\E$ if $\Span_AS$ is dense in $\E$. If there exists a finite or countable set generating $\E$ then we call $\E$ \emph{finitely generated} or \emph{countably generated}, respectively. If there exists a set $S$ such that $\Span_A S$ is equal to $\E$, then $\E$ is called \emph{algebraically finitely generated}.
   A Hilbert $A$-module is called \emph{projective} if it is projective as an $A$-module, namely if it is an (algebraic) direct summand in some free $A$-module.

One can show that if $\E$ is an algebraically finitely generated and projective $A$-module, then there exists an inner product on $\E$ which is unique up to isometry \cite{Waldmann:2021}. This inner product makes $\E$ a right Hilbert $A$-module. This implies that if $\E$ is an algebraically finitely generated and projective $A$-module, then any inner product makes $\E$ into a Hilbert module, meaning that completeness is automatic. We also obtain that if two algebraically finitely generated and projective Hilbert $A$-modules are isomorphic as $A$-modules, then they are isomorphic as Hilbert $A$-modules. 
\begin{definition}
    A \emph{finite frame} is a finite collection of elements $\zeta_1,...,\zeta_n$ of $\E$ such that for every $\eta$ in $\E$ we have $\eta=\sum_{i=1}^n\theta_{\zeta_i,\zeta_i}\eta$. A \emph{countable frame} is a sequence of elements $(\zeta_i)_{i=1}^\infty$ such that for every $\eta$ in $\E$ we have $\eta=\sum_{i=1}^\infty\theta_{\zeta_i,\zeta_i}\eta$ where the series converges in norm. 
\end{definition}
\begin{remark}
    What we call a (finite or countable) frame is called \emph{basis} in \cite{Kajiwara:2004} and \emph{standard normalized tight frame} in \cite{Frank:2002}. That the definition in \cite{Frank:2002} agrees with ours follows from \cite[Theorem 4.1]{Frank:2002}. 
\end{remark}
A Hilbert $A$-module which has a finite or countable frame is evidently algebraically finitely generated or countably generated, respectively. The converse also holds, see \cite[Proposition 1.2]{Kajiwara:2004}:
If a Hilbert $A$-module is finitely or countably generated, where in the finitely (countably) generated case we require $A$ to be ($\sigma$-)unital, then it has a finite or countable frame, respectively. 
We collect the following lemma for later use.
\begin{lemma}\label{lem:frame induces approx unit in compacts}
    Let $I$ be a finite or countable index set, and let $(\zeta_i)_{i\in I}$ be a finite or countable frame. Then $(\sum_{i\in F}\theta_{\zeta_i,\zeta_i})_{F\subset I}$, where $F$ ranges over all finite subsets of $I$, is an approximate unit for $\K(\E)$.
\end{lemma}
\begin{proof}
     Rank-one operators densely span $\K(\E)$. Hence it is sufficient to show that for all $\xi,\eta\in\E$ the expression $\sum_{i\in F}\theta_{\zeta_i,\zeta_i}\theta_{\xi,\eta}$ can be made arbitrarily close to $\theta_{\xi,\eta}$ by the right choice of a finite subset $F\subset I$. This follows from the inequality 
    \[\norm{\sum_{i\in F}\theta_{\zeta_i,\zeta_i}\theta_{\xi,\eta}-\theta_{\xi,\eta}}=\norm{\theta_{(\sum_{i\in F}\theta_{\zeta_i,\zeta_i}\xi-\xi),\eta}}\leq\norm{\sum_{i\in F}\theta_{\zeta_i,\zeta_i}\xi-\xi}\norm{\eta}.\]
    By the definition of a frame we can make the norm of $\sum_{i\in F}\theta_{\zeta_i,\zeta_i}\xi-\xi$ arbitrarily small. This finishes the proof.
\end{proof}

 To finish this section we briefly turn to quotients of Hilbert modules. The mentioned results can be found in \cite[Section 1]{Katsura2007}.

    Let $\E$ be a right Hilbert $A$-module, and let $I$ be an ideal in $A$. Then $\E I=\{\xi a:\xi\in\E,a\in I\}$ is a right Hilbert $I$-module with right multiplication and inner product inherited from $\E$. The quotient $\E_I\coloneqq \E/\E I$ is a right Hilbert $A/I$-module with right multiplication $[\xi]_I[a]_I=[\xi a]_I$ and inner product $\langle[\xi]_I,[\eta]_I\rangle_I=[\langle\xi,\eta\rangle]_I$ for $\xi,\eta\in\E$ and $a\in A$. 
    
\subsection{Hilbert \texorpdfstring{$C(X)$}{C(X)}-modules}\label{sect:Hilbert C(X)-modules}
The only Hilbert modules we will be concerned with in this paper, and one of the main examples in general, are Hilbert modules over commutative $C^*$-algebras. In this section, we only consider modules over unital commutative $C^*$-algebras, that is Hilbert $C(X)$-modules for a compact Hausdorff space $X$. The classical theorem of Serre--Swan \cite{Swan:1962} states that any finitely generated and projective $C(X)$-module arises as the space of continuous sections of some vector bundle over $X$. Therefore we first need to review the basic theory of vector bundles.

An introduction to vector bundles can be found in \cite{Husemoller:1993}. We denote a vector bundle over $X$ by $\mathcal{V}=[E,p,X]$ where $E$ is the total space, $X$ is the base space, and $p$ is a surjective continuous map from $E$ to $X$ called the bundle projection. All vector bundles in this work will be complex. 

A \emph{bundle morphism} between two vector bundles $[E_1,p_1,X]$ and $[E_2,p_2,X]$ over the same base space is a continuous map $f$ from $E_1$ to $E_2$ fixing the base space, that is we have $p_2\circ f=p_1$. Note that what we here call a bundle morphism is in \cite{Husemoller:1993} called a bundle morphism \emph{over} $X$. 

As is standard, we require our vector bundles to be locally trivial, meaning that for every point $x\in X$ there exist a neighborhood $U$ of $x$ and a bundle morphism $h$ from $U\times\C^{n_x}$ to $p^{-1}(U)$ inducing an isomorphism on each fiber, for some $n_x\in\N$. The pair $(U,h)$ is called a \emph{chart}, and a collection of charts such that the open sets cover $X$ is called an \emph{atlas}. Note that by definition of a chart, the rank of $\V$ is constant over $U$. More precisely, we have $\rank\V_y=n_x$ for all $y\in U$.

\emph{Sections} of a vector bundle $\V$ will be denoted by $\xi$, $\eta$ or $\zeta$. Spaces of continuous sections will play a very important role in the main part of this work. To put a norm on these spaces of sections we need a \emph{metric} on the vector bundle. A metric on $\V$ is a continuous map $\beta$ from the total space of $\V\oplus\V$ to $\C$ such that for each $x\in X$, restricting $\beta$ to $p^{-1}(x)\times p^{-1}(x)$ yields an inner product on $p^{-1}(x)$.

    Every complex vector bundle with paracompact base space has an hermitian metric, as is shown in \cite[Chapter 3]{Husemoller:1993}. An atlas $\{(U_i,h_i)\}$ is called \emph{metric preserving} if 
    \begin{equation}\label{eq:metric preserving}
        (v|w)=\beta(h_i(x,v),h_i(x,w))\quad\text{for every }x\in X \text{ and } v,w\in\C^{N_i},
    \end{equation}
    where $(\cdot|\cdot)$ is the standard inner product on $\C^{N_i}$. Metric preserving atlases have unitary transition functions.

 Let $\V_1=[E_1,p_1,X]$ and $\V_2=[E_2,p_2,X]$ be vector bundles over a space $X$. Let $\beta_1$ and $\beta_2$ be hermitian metrics on $\V_1$ and $\V_2$, respectively. A bundle morphism $f$ from $\V_1$ to $\V_2$ is called \emph{metric preserving} if $\beta_2(f(v),f(w))=\beta_1(v,w)$ for all $v$ and $w$ in $E_1$. If $f$ is metric preserving and bijective, then we call it an \emph{isometry}. 

The following result is certainly well known, but since we could not find a reference for it, we include the proof.
\begin{theorem}\label{thm:metric uniqueness}
    Let $\V$ be a vector bundle on a paracompact space. Let $\beta$ and $\Tilde{\beta}$ be two hermitian metrics on $\V$. Then there exists an isometry from $(\V,\beta)$ to $(\V,\Tilde{\beta})$. 
\end{theorem}
 \begin{proof}
     According to \cite[Chapter 5, Theorem 7.4]{Husemoller:1993} there exist metric-preserving atlases $\{(h_i,U_i)\}$ and $\{(\Tilde{h}_i,\Tilde{U}_i)\}$ associated to $\beta$ and $\Tilde{\beta}$, respectively. Let $g_{ij}$ and $\Tilde{g}_{ij}$ be systems of unitary transition functions for the atlases. We can assume that the open covers $\{U_i\}$ and $\{\Tilde{U}_i\}$ are the same, otherwise we use the open cover of intersections $U_i\cap \Tilde{U_j}$. 
     
     Since both atlases describe the same vector bundle, they are equivalent (see \cite[Chapter 5, Theorem 2.7]{Husemoller:1993}). Hence there exist functions $r_i:U_i\to GL(N_i)$ satisfying $\Tilde{g}_{ij}=r_i^{-1}g_{ij}r_j$. Using this equality and the fact that $g_{ij}^*=g_{ij}^{-1}=g_{ji}$, we obtain
     \begin{align*}   r_ir_i^*g_{ij}=r_i(g_{ij}^*r_i)^*=r_i(g_{ji}r_i)^*=r_i(r_j\tilde{g}_{ji})^*  =r_i\tilde{g}_{ij}r_j^* =g_{ij}r_jr_j^*.
     \end{align*}
     This implies $P(r_ir_i^*)g_{ij}=g_{ij}P(r_jr_j^*)$ for all polynomials $P$. Using functional calculus for the positive matrix $r_ir_i^*$ we obtain $|r_i|^{-1}g_{ij}=g_{ij}|r_j|^{-1}$. 

     Consider the polar decomposition $r_i=|r_i|u_i$, where $u_i=|r_i|^{-1}r_i$ is unitary. We have \[g_{ij}u_j=g_{ij}|r_j|^{-1}r_j=|r_i|^{-1}g_{ij}r_j=|r_i|^{-1}r_i\Tilde{g}_{ij}=u_i\Tilde{g}_{ij}.\]
     This allows us to define a bundle morphism $f$ on $\V$ in the same way as in the proof of \cite[Chapter 5, Theorem 2.7]{Husemoller:1993}. We first define $f_i:U_i\times\C^{N_i}\to U_i\times\C^{N_i}$ by $f_i(x,v)=(x,u_i^*v)$. Let $f$ be the map on $\V$ that restricted to $p^{-1}(U_i)$ is equal to $\Tilde{h}_i\circ f_i\circ h_i^{-1}$. This is well defined: We have \begin{align*}
         (\Tilde{h}_j\circ f_j\circ h_j^{-1})(h_i(x,v))&=(\Tilde{h}_j\circ f_j)(x,g_{ji}v))=\Tilde{h}_j(x,u_j^*g_{ji}v)\\&=\Tilde{h}_j(x,\Tilde{g}_{ji}u_i^*v)=\Tilde{h}_i(x,u_i^*v)=(\Tilde{h}_i\circ f_i\circ h_i^{-1})(h_i(x,v)).
     \end{align*}
     Since $f$ is locally a bijection it is an isomorphism of vector bundles. To show that $f$ is an isometry, we use (\ref{eq:metric preserving}) to calculate
     \begin{align*}
         \Tilde{\beta}(f(h_i(x,v)),f(h_i(x,w)))&=\Tilde{\beta}(\Tilde{h}_i\circ f_i(x,v),\Tilde{h}_i\circ f_i(x,v))\\&=\Tilde{\beta}(\Tilde{h}_i(x,u_i^*v),\Tilde{h}_i(x,u_i^*v))=(u_i^*v|u_i^*w)=(v|w)\\&=\beta(h_i(x,v),h_i(x,w)),
     \end{align*}
     which concludes the proof.
 \end{proof}
Let $\V=[E,p,X]$ be a vector bundle on a compact Hausdorff space $X$. We write $\Gamma(\V)$ for the vector space of continuous sections of $\V$, meaning continuous maps $\xi:X\to E$ such that $p\circ \xi=\id_X$. The space of sections $\Gamma(\V)$ is a right $C(X)$-module, where multiplication of a function $f$ with a section $\xi$ is defined as $(\xi\cdot f)(x)=\xi(x)f(x)$. The theorem of Serre--Swan \cite{Swan:1962} asserts that $\Gamma(\V)$ is algebraically finitely generated and projective and that every algebraically finitely generated projective $C(X)$-module arises in this way. In fact, sending a vector bundle $\mathcal{V}$ to its space of continuous sections $\Gamma(\mathcal{V})$, establishes an equivalence of categories from the category of vector bundles over $X$ with bundle morphisms to the category of algebraically finitely generated projective $C(X)$-modules with module morphisms. 

We can use the unique metric $\beta$ to define an inner product on $\Gamma(\V)$ by the formula 
\[\langle\xi,\eta\rangle(x)=\beta(\xi(x),\eta(x))\quad\text{for }x\in X,\]
where $\xi$ and $\eta$ are elements of $\Gamma(\V)$. Since $\Gamma(\V)$ is finitely generated and projective it will be complete with respect to this inner product, that is, $\Gamma(\V)$ is a Hilbert $C(X)$-module. Sending a vector bundle to its space of continuous sections now yields an equivalence of categories between the category of vector bundles equipped with hermitian metrics, and metric preserving bundle morphisms on one hand, and the category of algebraically finitely generated and projective right Hilbert $C(X)$-modules with Hilbert module morphisms on the other hand.

We know that $\Gamma(\V)$ has a finite frame, since it is algebraically finitely generated. We can however also construct a finite frame explicitly: Let $(U_i,h_i)_{i=1}^m$ be a metric-preserving atlas of $\V$, where as usual $h_i:U_i\times\C^{n_i}\to p^{-1}(U_i)$ is a fiber-preserving homeomorphism. Let $\{\gamma_i\}_{i=1}^m$ be a partition of unity subordinate to the open cover $\{U_i\}_{i=1}^m$, and denote the $k$-th standard basis vector of $\C^{n_i}$ by $e_k$. Define a section $\xi_{i,k}$ by $\xi_{i,k}(x)\coloneqq h_i(x,\gamma_i(x) e_k)$.

\begin{lemma}\label{lem:frame vector bundle}
    The collection of all $\xi_{i,k}$ forms a finite frame.
\end{lemma}

\subsection{\texorpdfstring{$C^*$}{C*}-correspondences and Hilbert bimodules}\label{sec:hilbert bimodules}
A \emph{$C^*$-correspondence} over a $C^*$-algebra $A$ is a right Hilbert $A$-module $\E$ together with a $\ast$-homomorphism $\varphi_\E$ from $A$ to $\calL(\E)$.

We call $\varphi_\E$ the \emph{structure map}. A $C^*$-correspondence $\E$ is called \emph{non-degenerate} if \\$\overline{\Span_A}(\varphi_\E(A)\E)$ is equal to $\E$. Here $\overline{\Span_A}(S)$ denotes the closure of the $A$-linear span of a set $S \subset \E$. The basics about $C^*$-correspondences can be found in \cite{Katsura:2004}.

\begin{example}\label{ex:homomorphisms as correspondences}
    Let $A$ be a unital $C^*$-algebra and $\phi:A\to A$ a $\ast$-homomorphism. Equip $A$ with algebra multiplication as right multiplication and the inner product $\langle a,b\rangle_A=a^*b$. Then $A$ is a right Hilbert $A$-module and $\mathcal{K}(A)$ is isomorphic to $A$. The $\ast$-homomorphism $\phi$ induces a structure map $\varphi_A:A\to\mathcal{K}(A)$ given by $\varphi_A(a)b=\phi(a)b$. This yields a $C^*$-correspondence over $A$. If $\phi=\id_A$ then we denote the resulting $C^*$-correspondence by $\mathbf{A}$. 
\end{example}
Let $\E$ and $\F$ be $C^*$-correspondences over $A$. We denote the internal tensor product of $\E$ and $\F$ by $\E\otimes_A\F$. It is a $C^*$-correspondence over $A$ with  inner product \[\langle\xi_1\otimes\eta_1,\xi_2\otimes\eta_2\rangle_{\E\otimes\F}\coloneqq \langle\eta_1,\varphi_\F(\langle\xi_1,\xi_2\rangle_\E)\eta_2\rangle_\F.\] The $n$-fold tensor product of $\E$ with itself is denoted by $\E^{\otimes n}$.

If $J$ is an ideal in a $C^*$-algebra $A$, write $J^\perp$ for the orthogonal complement $\{a\in A:aJ=Ja=\{0\}\}$.
\begin{definition}[{\cite[Definition 3.2]{Katsura:2004}}]\label{def:Katsuras ideal}
    Let $\E$ be a $C^*$-correspondence from $A$ to $B$ with structure map $\varphi_\E$. We define the ideal \[J_\E\coloneqq\varphi_\E^{-1}(\K(\E))\cap(\ker\varphi_\E)^\perp\unlhd A,\] and call it \emph{Katsura's ideal}. 
\end{definition}
We now turn to Hilbert bimodules, which are at the same time right and left Hilbert modules with compatible multiplications and inner products. 
\begin{definition}\label{def:bimodule}
    Let $A$ be a $C^*$-algebra. Let $\E$ be a right Hilbert $A$-module with inner product $\langle\cdot,\cdot\rangle_\E$. If $\E$ is also a left Hilbert $A$-module with inner product ${}_\E\langle\cdot,\cdot\rangle$ satisfying
    \[a(\xi b)=(a\xi)b\quad\text{and}\quad{}_\E\langle\xi,\eta\rangle\zeta=\xi\langle\eta,\zeta\rangle_\E\]
     for all $\xi,\eta,\zeta\in\E$ and $a,b\in A$, then we call $\E$ a \emph{Hilbert $A$-bimodule}.
\end{definition}
Every Hilbert $A$-bimodule is a $C^*$-correspondence over $A$. To see this, define the structure map $\varphi_\E$ to act by left multiplication, $\varphi_\E(a)\xi=a\xi$. It is shown in \cite{Brown:1994} that $\varphi_\E$ indeed maps into the adjointable operators on $\E$. If we restrict $\varphi_\E$ to Katsura's ideal $J_\E$ then it becomes a $\ast$-isomorphism onto $\K(\E)$. We can also go the opposite direction: Any $C^*$-correspondence such that the structure map restricts to an isomorphism from Katsura's ideal to the compacts can be turned into a Hilbert $A$-bimodule. Left multiplication is given by $a\xi\coloneqq \varphi_\E(a)\xi$, and the inner product by ${}_\E\langle\xi,\eta\rangle\coloneqq\varphi_\E^{-1}(\theta_{\xi,\eta})$.
\begin{example}
Let $X$ be a compact Hausdorff space, $\V$ a vector bundle over $X$, and $\alpha:X\to X$ a homeomorphism. We can turn the right Hilbert $C(X)$-module $\Gamma(\V)$ into a $C^*$-correspondence by defining the left action $\varphi(f)\xi=\xi(f\circ\alpha)$ for $\xi\in\Gamma(\V)$ and $f\in C(X)$. This $C^*$-correspondence, which we denote by $\Gamma(\V,\alpha)$, is studied in \cite{adamo2023}. It is shown there that $\Gamma(\V,\alpha)$ is a Hilbert bimodule if and only if $\V$ is a line bundle. The main motivation for this work was to generalize the results of \cite{adamo2023} by taking a partial automorphism $\theta$ on $X$ instead of a homeomorphism $\alpha$.
\end{example}
\subsection{Morphisms and representations of \texorpdfstring{$C^*$}{C*}-correspondences}
Throughout this section $\E$ and $\F$ will denote $C^*$-correspondences over $C^*$-algebras $A$ and $B$, respectively. 
\begin{definition}\label{def:morphism Cstar corr}
    A \emph{morphism of $C^*$-correspondences} from $\E$ to $\F$ is a morphism of Hilbert modules $(\Pi : A \to B , \, T : \E \to \F)$, see Definition \ref{def:morphism of Hilbert modules}, such that \begin{equation}\label{eq:morphism left action}\varphi_\F(\Pi(a))T(\xi)=T(\varphi_\E(a)\xi)\quad\text{for all } a\in A,\xi\in\E.\end{equation}

Given a morphism $(\Pi,T)$ of $C^*$-correspondences from $\E$ to $\F$, the map $\psi_T:\K(\E)\to\K(\F)$ defined by $\psi_T(\theta_{\xi,\eta})=\theta_{T(\xi),T(\eta)}$, for $\xi,\eta\in\E$, is a $^*$-homomorphism. 
\end{definition}
\begin{definition}
A morphism $(\Pi,T)$ of $C^*$-correspondences from $\E$ to $\F$ is called \emph{covariant} if $\Pi(a)\in J_\F$ and $\varphi_\F(\Pi(a))=\psi_T(\varphi_\E(a))$ for all $a\in J_\E$. 
\end{definition}
\begin{lemma}\label{lem:bijective implies covariant} 

    Let $(\Pi,T)$ be a morphism of $C^*$-correspondences from $\E$ to $\F$. If $\Pi$ is surjective and $T$ is bijective, then $(\Pi,T)$ is covariant.
    \end{lemma}
\begin{proof}
    By \cite[Proposition 2.4]{Robertson2011} the morphism $(\Pi,T)$ is covariant if $\Pi(J_\E)\subset J_\F$ and $T$ has dense range. 
   We will now show that $\Pi$ being surjective and $T$ being injective implies that $\Pi((\ker\varphi_\E)^\perp)$ is contained in $(\ker\varphi_\F)^\perp$. As an intermediate step we prove that $\ker\varphi_\F$ is contained in $\Pi(\ker\varphi_\E)$. Since $\Pi$ is surjective, any element of $B$ can be written as $\Pi(a)$ for some $a\in A$. If $\Pi(a)$ lies in $\ker\varphi_\F$ then we obtain \[T(\varphi_\E(a)\xi)=\varphi_\F(\Pi(a))T(\xi)=0\quad\text{for all }\xi\in\E,\]
    so if $T$ is injective then this implies $\varphi_\E(a)=0$ and hence $a\in\ker\varphi_\E$.
    
    Note that $a\in(\ker\varphi_\E)^\perp$ implies $\Pi(a)\Pi(b)=0$ for all $b\in\ker\varphi_\E$. Since $\ker\varphi_\F$ is contained in $\Pi(\ker\varphi_\E)$ this shows that $\Pi(a)c=0$ for all $c\in\ker\varphi_\F$, that is, $\Pi(a)\in(\ker\varphi_\F)^\perp$.

    We have left to show that $\Pi(\varphi_\E^{-1}(\K(\E)))$ is contained in $\varphi_\F^{-1}(\K(\F))$. But this is clear from surjectivity of $T$ together with Equation \ref{eq:morphism left action}. It now follows from Definition \ref{def:Katsuras ideal} that $\Pi(J_\E)$ is contained in $J_\F$.
    \end{proof}
\begin{definition}\label{def:subcorrespondence}
Let $\F$ be a $C^*$-correspondence over a $C^*$-algebra $B$. A \emph{$C^*$-subcorrespondence of $\F$} is a $C^*$-correspondence $\E$ over a $C^*$-algebra $A$ together with an injective covariant morphism from $\E$ to $\F$.
\end{definition}
If the morphism is understood we will not denote it explicitly, and speak of $\E$ as a $C^*$-subcorrespondence of $\F$. 

In the following definition we will use the notation of Example \ref{ex:homomorphisms as correspondences}. 
\begin{definition}\label{def: covariant representation}
    Let $\E$ be a $C^*$-correspondence over a $C^*$-algebra $A$. A \emph{covariant representation} of $\E$ on a $C^*$-algebra $B$ is a covariant morphism of $C^*$-correspondences from $\E$ to $\mathbf{B}$.
\end{definition}
Explicitly, a representation of a $C^*$-correspondence $\E$ on a $C^*$-algebra $B$ is given by a pair $(\pi,t)$. The map $\pi$ is a $\ast$-homomorphism from $A$ to $B$, and $t$ is a linear map from $\E$ to $B$ such that $\pi(\langle\xi,\eta\rangle_\E)=t(\xi)^*t(\eta)$ and $\pi(a)t(\xi)=t(\varphi_\E(a)\xi)$ for every $\xi,\eta\in\E$ and every $a\in A$. This automatically implies that $t(\xi)\pi(a)=t(\xi a)$ holds. 

The $\ast$-homomorphism $\psi_t$ from Definition \ref{def:morphism Cstar corr} now goes into the $C^*$-algebra $B$, and on rank-one operators is given by $\psi_t(\theta_{\xi,\eta})=t(\xi)t(\eta)^*$.

\subsection{Cuntz--Pimsner algebras}
We will now define the Cuntz--Pimsner algebra $\CP(\E)$ associated to a $C^*$-correspondence $\E$. Cuntz--Pimsner algebras were first introduced by Pimsner in \cite{Pimsner:1997} and then defined by Katsura in \cite{Katsura:2004} in the most general setting. 

Let $\E$ be a $C^*$-correspondence over a $C^*$-algebra $A$. If $(\pi,t)$ is a covariant representation of $\E$ on a $C^*$-algebra $B$, then $C^*(\pi,t)$ denotes the $C^*$-subalgebra of $B$ generated by the images of $\pi$ and $t$. 

The covariant representation $(\pi,t)$ is called \emph{universal} if for any other covariant representation $(\Tilde{\pi},\Tilde{t})$ there exists a surjective $\ast$-homomorphism $\rho$ from $C^*(\pi,t)$ onto $C^*(\Tilde{\pi},\Tilde{t})$ such that $\Tilde{\pi}=\rho\circ\pi$ and $\Tilde{t}=\rho\circ t$. We denote the universal covariant representation of $\E$ by $(\pi_\E,t_\E)$. The $C^*$-algebra $C^*(\pi_\E,t_\E)$ is called the \emph{Cuntz--Pimsner algebra} and is denoted by $\CP(\E)$.
 
By definition the universal covariant representation is unique if it exists. To show that it does exist one has to construct it explicitly, which also shows that the Cuntz--Pimsner algebra of $\E$ exists and is unique. This is done in \cite{Katsura:2004}. 

A representation $(\pi,t)$ of $\E$ is said to \emph{admit a gauge action} if for every $z\in\mathbb{T}$ there exists a $\ast$-homomorphism $\beta_z$ on $C^*(\pi,t)$ such that $\beta_z(t(\xi))=zt(\xi)$ and $\beta_z(\pi(a))=\pi(a)$ for all $\xi\in\E$ and $a\in A$.

The universal covariant representation admits a gauge action which we denote by $\gamma$. The gauge-invariant uniqueness theorem \cite[Theorem 6.4]{Katsura:2004}, which is fundamental for the theory of Cuntz--Pimsner algebras, says that for a covariant representation $(\pi,t)$ of a $C^*$-correspondence $\E$, the $\ast$-homomorphism $\rho:\CP(\E)\to C^*(\pi,t)$ is an isomorphism if and only if $(\pi,t)$ is injective and admits a gauge action. Here we say that $(\pi,t)$ is injective if $\pi$ is injective, which automatically implies injectivity of $t$. The gauge-invariant uniqueness theorem allows one to show that if $\E$ is a $C^*$-subcorrespondence of $\F$, then $\CP(\E)$ is a subalgebra of $\CP(\F)$, see \cite[Remark 3.7]{Katsura2007} and \cite[Lemma 2.6]{Katsura:2004}.

\begin{lemma}\label{lem:inductive limit correspondences}
    Let $\E$ and $\E_n$ for every $n\in\N$ be $C^*$-correspondences over $C^*$-algebras $A$ and $A_n$, respectively. Let $(\Pi_n,T_n):\E_n\to \E_{n+1}$ and $(\Pi^{(n)},T^{(n)}):\E_n\to\E$ be injective covariant morphisms such that $(\Pi^{(n)},T^{(n)})=(\Pi^{(n+1)}\circ\Pi_n,T^{(n+1)}\circ T_n)$ for all $n\in\N$. Additionally assume
    \begin{align*}
        A=\overline{\bigcup_{n\in\N}\Pi^{(n)}(A_n)}\qquad\text{and}\qquad  \E=\overline{\bigcup_{n\in\N}T^{(n)}(\E_n)}. 
    \end{align*}
    Then $\CP(\E)=\lim\limits_{\to}\CP(\E_n)$.
\end{lemma}
\begin{proof}
    Since assigning its Cuntz--Pimsner algebra to a $C^*$-correspondence gives a covariant functor which preserves monomorphisms (see \cite[Proposition 2.11]{Robertson2011}), we obtain injective \mbox{*-homomorphisms} $\varphi_n:\CP(\E_n)\to\CP(\E_{n+1})$ and $\varphi^{(n)}:\CP(\E_n)\to\CP(\E)$ such that $\varphi^{(n)}=\varphi^{(n+1)}\circ\varphi_n$ for all $n\in\N$. By the universal property of the inductive limit, this yields a *-homomorphism $\psi:\lim\limits_{\to}\CP(\E_n)\to\CP(\E)$, which is injective because all $\varphi^{(n)}$ are injective. To show that it is surjective, we need to show
    \begin{align}\label{eq:inductive limit}
        \CP(\E)=\overline{\bigcup_{n\in\N}\varphi^{(n)}(\CP(\E_n))}.
    \end{align}
    By assumption, we can approximate elements of $A$ by elements of $\bigcup_{n\in\N}\Pi^{(n)}(A_n)$, and elements of $\E$ by elements of $\bigcup_{n\in\N}T^{(n)}(\E_n)$. If we regard elements of $A$, $\E$ and the $A_n$'s and $\E_n$'s as elements of $\CP(\E)$ in the canonical way, this means that we can approximate elements of $A$ by elements of  $\bigcup_{n\in\N}A_n\subset \bigcup_{n\in\N}\varphi^{(n)}(\CP(\E_n))$ in the Cuntz--Pimsner algebra, and elements of $\E$ by elements of $\bigcup_{n\in\N}\E_n\subset \bigcup_{n\in\N}\varphi^{(n)}(\CP(\E_n))$. Considering that multiplication and addition in $\CP(\E)$ are continuous, and that the *-subalgebra generated by $\E$ and $A$ is dense in $\CP(\E)$, this shows (\ref{eq:inductive limit}). 
\end{proof}     

\section{Cuntz--Pimsner algebras of partial automorphisms twisted by vector bundles}\label{sect:cuntz pimsner algebras from twisted partial autos}
In this section we build a $C^*$-correspondence starting from a partial automorphism twisted by a vector bundle, by which we mean a partial action of the integers on the base space of a vector bundle. The construction is almost identical to the one in \cite{adamo2023}. We investigate the properties of the associated Cuntz--Pimsner algebras. 

 \subsection{Hilbert \texorpdfstring{$C_0(X)$}{C0(X)}-modules}\label{sec:Hilbert Czero modules}
The first step towards constructing a $C^*$-correspondence out of a partial $\Z$-action and a vector bundle is to associate a Hilbert module to the vector bundle. In Section \ref{sect:Hilbert C(X)-modules} we have considered Hilbert $C(X)$-modules for a compact Hausdorff space $X$. We now deal with locally compact spaces as well. More precisely, we will construct Hilbert $C_0(X)$-modules from vector bundles over locally compact second countable Hausdorff spaces. We will then study how they relate to general Hilbert $C_0(X)$-modules. This type of module has, for example,  been considered in \cite{robertTikuisis:2011}. The main result of the section is a Serre--Swan type theorem for a subclass of Hilbert $C_0(X)$-modules which we call \emph{$\sigma$-finitely generated} and \emph{$\sigma$-projective}. Another generalization of the Serre--Swan theorem, which goes into a different direction by considering Hilbert bundles instead of vector bundles, can be found in \cite{Takahashi:1979}.

 Let $\V=[E,p,X]$ be a vector bundle on the locally compact Hausdorff space $X$. According to Theorem \ref{thm:metric uniqueness} there exists a metric $\beta$ on $\V$ that is unique up to isometry. We denote by $\Gamma_0(\V)$ the space of continuous sections of $\V$ vanishing at infinity with respect to $\beta$. 
 \begin{proposition}
     The space $\Gamma_0(\V)$ with right multiplication $(\xi f)(x)=\xi(x)f(x)$ by continuous functions vanishing at infinity and the inner product  $\langle\xi,\eta\rangle(x)=\beta(\xi(x),\eta(x))$ is a right Hilbert $C_0(X)$-module. 
 \end{proposition}
\begin{proof}
    The space $\Gamma_0(\V)$ is invariant under the right multiplication given above. That $\langle\cdot,\cdot\rangle$ defined as above is $C_0(X)$-valued follows because if both $x\mapsto\beta(\xi(x),\xi(x))$ and $x\mapsto\beta(\eta(x),\eta(x))$ vanish at infinity, then $x\mapsto\beta(\xi(x),\eta(x))$ does as well. This can be seen by applying the Cauchy--Schwartz inequality fiberwise. The axioms of the inner product can be verified directly using that $\beta$ defines an inner product on each fiber. Hence $\Gamma_0(\V)$ equipped with $\langle\cdot,\cdot\rangle$ is an inner product $C_0(X)$-module. 

    It is left to show that $\Gamma_0(\V)$ is complete with respect to $\langle\cdot,\cdot\rangle$. Let $(\xi_n)$ be a sequence in $\Gamma_0(\V)$ which is Cauchy with respect to $\langle\cdot,\cdot\rangle$. Let $K$ be a compact subset of $X$. Write $\V_K$ for the restriction of $\V$ to $K$. There exists a surjective and contractive morphism from $\Gamma_0(\V)$ to $\Gamma(\V|_K)$ given by restriction, which we call $\Phi_K$. Then $(\Phi_K(\xi_n))_n$ is a Cauchy sequence in $\Gamma(\V_K)$, which has a limit $\xi^{K}$. If $K$ and $L$ are two compact subsets of $X$, then $\xi^K$ and $\xi^L$ agree on the intersection $K\cap L$. Hence we can define a continuous section $\xi$ such that the restriction of $\xi$ to $K$ yields $\xi^{K}$. 
    
    The sequence $(\xi_n)$ converges to $\xi$ uniformly on compact sets. We want to show that it converges in the norm induced by $\langle\cdot,\cdot\rangle$. Take $\varepsilon>0$. There exists $n_0\in\N$ such that $\norm{\xi_n-\xi_m}\leq\varepsilon/4$ for all $m,n>n_0$, and there exists a compact set $K$ such that $\beta(\xi_{n_0}(x),\xi_{n_0}(x))^{1/2}\leq\varepsilon/4$ for all $x\in X\backslash K$. It follows that $\beta(\xi_{n}(x),\xi_{n}(x))^{1/2}\leq\varepsilon/2$ for all $x\in X\backslash K$ and $n>n_0$. The same holds for $\xi$ instead of $\xi_n$, since $(\xi_n)$ converges to $\xi$ pointwise. In particular $\xi$ lies in $\Gamma_0(\V)$.
    
    Since $(\Phi_{K}(\xi_k))_k$ converges to $\xi^{K}$, there exists $k_0>n_0$ such that $\norm{\Phi_{K}(\xi_k)-\xi^{K}}\leq\varepsilon$ for all $k>k_0$. Put into words, this means that for all $k>k_0$ the restriction of $\xi_k$ to $K$ is $\varepsilon$-close to the restriction of $\xi$ to $K$. Since outside of $K$ both $\xi_k$ and $\xi$ are bounded by $\varepsilon/2$, this shows $\norm{\xi_{k}-\xi}\leq\varepsilon$. Therefore $(\xi_n)$ converges to $\xi\in\Gamma_0(\V)$.    
\end{proof}
\begin{remark}\label{rem:sigma fp module is full}
    Let $\V$ be a vector bundle over $X$. Recall that we assume the bundle projection to be surjective, and each fiber to be at least one-dimensional. Using local trivialisations, one can show that each point of $X$ has an open neighborhood $W$ such that $C_0(W)$ is contained in $\langle\Gamma_0(\V),\Gamma_0(\V)\rangle$. A partition-of-unity argument implies that $\langle\Gamma_0(\V),\Gamma_0(\V)\rangle$ is equal to $C_0(X)$, In other words, $\Gamma_0(\V)$ is full. 
\end{remark}
 Recall from Section \ref{sect:Hilbert modules} that if $\E$ is a right Hilbert $A$-module and $I$ is an ideal in $A$, then we write $\E_I$ for the quotient Hilbert $A/I$-module $\E/\E I$.
\begin{definition}
    A right Hilbert $C_0(X)$-module $\E$ is called \emph{$\sigma$-finitely generated} if for any compact subset $K$ of $X$, the Hilbert $C(K)$-module $\E_{C_0(X\backslash K)}$ is finitely generated. It is called \emph{$\sigma$-projective} if $\E_{C_0(X\backslash K)}$ is projective. 
\end{definition}

The following theorem is a version of the Serre--Swan theorem for $\sigma$-finitely generated and $\sigma$-projective Hilbert $C_0(X)$-modules.
\begin{theorem}\label{thm:serre swan locally compact}
    Let $X$ be a locally compact Hausdorff space. For any vector bundle $\V$ over $X$ and any metric $\beta$ on $\V$, the right Hilbert $C_0(X)$-module $\Gamma_0(\V)$ is $\sigma$-finitely generated and $\sigma$-projective. Conversely, every full $\sigma$-finitely generated and $\sigma$-projective Hilbert $C_0(X)$-module is isomorphic to a module of the form $\Gamma_0(\V)$ for some vector bundle $\V$ and metric $\beta$. 
\end{theorem}
\begin{proof}

    That $\Gamma_0(\V)$ is $\sigma$-finitely generated and $\sigma$-projective follows from the fact that the quotient module $\Gamma_0(\V)_{C_0(X\backslash K)}$ is isomorphic to $\Gamma(\V|_{K})$ together with the Serre--Swan theorem. 

    Conversely, let $\E$ be a full $\sigma$-finitely generated and $\sigma$-projective Hilbert $C_0(X)$-module. Let $K$ be a compact subset of $X$. To ease notation we write $\E_K$ for the quotient $\E_{C_0(X\backslash K)}$ of $\E$ by $\E C_0(X\backslash K)$. Using the Serre--Swan theorem yields a vector bundle $\V_K=[E_K,p_K,K]$ over $K$ such that $\E_K$ and $\Gamma(\V_K)$ are isomorphic as right Hilbert $C(K)$-modules. 
    
    Consider the disjoint union 
    \[\Tilde{E}\coloneqq \bigsqcup_{x\in X}\E_x\]
    as a set, and define the obvious projection $\Tilde{p}$ from $\Tilde{E}$ onto $X$. The surjection from $\E$ onto $\E_K$ yields a vector space isomorphism between the fibers $\E_x$ and $p_K^{-1}(x)$, for all $x\in K$. We thus obtain fiber-preserving maps $\phi_K$ from $E_K$ to $\Tilde{E}$, which are bijections onto their image $\Tilde{p}^{-1}(K)$. Equip $\Tilde{E}$ with the final topology with respect to the family of maps $\{\phi_K:K\subset X\text{ compact}\}$, and call the resulting topological space $E$. The projection $\Tilde{p}$ gives rise to a continuous map $p$ from $E$ to $X$. This is because for all compact $K$, the concatenation $p\circ\phi_K$ is equal to the bundle projection $p_K$, which is continuous. It then follows from the properties of the final topology that $p$ is continuous as well. 

    We obtain a triple $\V\coloneqq[E,p,X]$ with $p:E\to X$ a continuous map. By definition the fiber $p^{-1}(x)$ over any point $x\in X$ is a complex vector space. In order to show that $\V$ is a vector bundle, we need to verify that it is locally trivial. To this end, observe that the bijections $\phi_K$ from above give rise to homeomorphisms from $E_K$ to $p^{-1}(K)$, which we also call $\phi_K$. Local triviality of $\V$ now follows from the local triviality of the $\V_K$. 

    We now show that the space of continuous sections of $\V$ vanishing at infinity, with respect to the metric on $\V$ induced fiberwise by the inner product on $\E$, is isomorphic to $\E$. An element $\xi$ of $\E$ induces a function $\Tilde{\xi}$ from $X$ to $\Tilde{E}$ such that $\Tilde{p}\circ\Tilde{\xi}=\id_X$. Using that the $\phi_K$ are homeomorphisms, we see that $\Tilde{\xi}$ gives rise to a continuous section of $\V$. This section vanishes at infinity by definition of the metric on $\V$. Thus there is a linear module map from $\E$ to $\Gamma_0(\V)$, which is in fact isometric. All that remains to show is surjectivity of this map. Take an element $\xi$ of $\Gamma_0(\V)$. For every $n\in\N$ define the compact set $K_n$ containing all points $x\in X$ such that the norm of $\xi(x)$ is greater or equal then $1/n$. Thus we obtain a sequence $(K_n)$ of compact subsets of $X$ such that $K_n$ is contained in the interior of $K_{n+1}$, for all $n\in\N$. Take a sequence $(\gamma_n)$ of continuous functions such that $\gamma_n$ is equal to one on $K_n$, and equal to zero outside the interior of $K_{n+1}$. Then $\gamma_n\xi$ is an element of $\Gamma(\V_{K_{n+1}})\cong\E_{K_{n+1}}$ whose support is contained in the interior of $K_{n+1}$. Thus we can regard $\gamma_n\xi$ as an element of $\E$. We obtain a Cauchy sequence $(\gamma_n\xi)$ inside of $\E$. Since $\E$ is complete, this sequence has a limit $\xi'$, whose value in each fiber is equal to that of $\xi$. Hence our isometry from $\E$ to $\Gamma_0(\V)$ maps $\xi'$ to $\xi$. This shows that it is surjective, and thus an isomorphism. 
    
    \end{proof}

  One can go a step further and show that there is an equivalence of categories between the category of vector bundles equipped with hermitian metrics and metric preserving bundle morphisms over $X$ on one hand, and the category of full $\sigma$-finitely generated and $\sigma$-projective Hilbert $C_0(X)$-modules with Hilbert module morphisms on the other hand. This is analogous to the equivalence of categories established by the classical Serre--Swan theorem. 

We now investigate the structure of the $C^*$-algebra of compact operators $\K(\E)$ of a $\sigma$-finitely generated and $\sigma$-projective Hilbert $C_0(X)$-module. First some notation: We write $I_x$ for $C_0(X\backslash\{x\})$, and $\E_x$ for the quotient $\E/\E I_x$. Then $\E_x$ is a Hilbert $\C$-module, and hence a Hilbert space.

We define a $\ast$-homomorphism $R:C_0(X)\to\calL(\E)$ by $R(f)\xi=\xi f$. It is shown in \cite[Section 3.2]{boenicke:2018} that $R$ maps into the center of $\calL(\E)\cong M(\K(\E))$ and is non-degenerate. By definition, this gives $\K(\E)$ the structure of a $C_0(X)$-algebra. The fiber $\K(\E)_x$ at a point $x\in X$ is isomorphic to $\K(\E_x)$, see \cite[Page 72]{boenicke:2018}. The algebra of adjointable operators, being the multiplier algebra of $\K(\E)$, fibers over $X$ as well, but it is not a $C_0(X)$-algebra (see \cite{Akemann:1973}).

\begin{proposition}\label{cor:compacts for sigma fgp module}
    Let $\E$ be a $\sigma$-finitely generated and $\sigma$-projective Hilbert $C_0(X)$-module. Then the compact operators $\K(\E)$ are a continuous $C_0(X)$-algebra, and \[\K(\E)=\{t\in\calL(\E):(x\mapsto\norm{t_x})\in C_0(X)\}.\] 
\end{proposition}
\begin{proof}
    Since the vector bundle is locally trivial, the compacts are as well, meaning that around every point $x$ there exists a neighborhood $W$ such that $R(C_0(W))\K(\E)$ is isomorphic to $C_0(W)\otimes M_{d(x)}$ as $C_0(W)$-algebras. We obtain that $\K(\E)$ is a continuous $C_0(X)$-algebra.

    The second statement follows from the fact that $\K(\E)$ is generated by rank-one operators $\theta_{\xi,\eta}$ and $\xi$ as well as $\eta$ vanish at infinity. 
  
\end{proof}
At the end of this section we establish two technical lemmas under the additional assumption that the space $X$ is second countable. 

 Let $X$ be a locally compact second countable Hausdorff space. In particular $X$ is paracompact and exhaustible by compact sets. The latter means that there exists a sequence of compact sets $K_n$ such that $K_n\subset \interior(K_{n+1})$ and $\bigcup_n K_n=X$. 

\begin{lemma}\label{lem:inductive limit exhaustion by compact sets}
    Let $(K_n)$ be a sequence of compact subsets exhausting $X$, and let $\E$ be a $C_0(X)$-module. Then $\bigcup_n\E C_0(\interior K_n)$ is dense in $\E$. 
\end{lemma}
\begin{proof}
    Since the $K_n$ exhaust $X$ we can take an approximate unit $(\chi_n)_{n\in\N}$ of $C_0(X)$ such that $\chi_n\in C_0(\interior K_n)$ for all $n\in\N$. Then $\xi \chi_n\to \xi$ in norm for all $\xi\in\E$. This shows the claim.
\end{proof}
 
\begin{lemma}\label{lem:frame for submodule}
    Let $\E$ be a $\sigma$-finitely generated and $\sigma$-projective Hilbert $C_0(X)$-module. For every compact subset $K$ of $X$ there exists $N\in\N$ and a collection $\{\xi_n\}_{n=1}^N$ of elements of $\E$ such that 
    \[\sum_{n=1}^N\theta_{\xi_n,\xi_n}\eta=\eta,\] 
    for all $\eta\in\E C_0(\interior K)$. This implies that $\E$ is countably generated. 
\end{lemma}
\begin{proof} Let $K$ be a compact subset of $X$.
    According to Theorem \ref{thm:serre swan locally compact} we can assume without loss of generality that $\E$ is of the form $\Gamma_0(\V)$ for some vector bundle $\V$ over $X$. Since $\Gamma(\V|_K)$ is finitely generated, it has a frame. Hence there exists $N\in\N$ and sections $\xi_1^{(0)},\xi_2^{(0)},...,\xi_N^{(0)}$ in $\Gamma(\V|_K)$ such that 
    \[\sum_{n=1}^N\theta_{\xi_n^{(0)},\xi_n^{(0)}}\eta=\eta,\] 
    for all $\eta\in\Gamma(\V|_K,\beta|_K)$. By \cite[Theorem 5.7]{Karoubi:1978} we can extend the $\xi_n^{(0)}$ to sections $\xi_n\in\Gamma_0(\V)$. Note that \cite[Theorem 5.7]{Karoubi:1978} is only stated for $X$ compact, but the adaptation to locally compact spaces is straightforward. The claim follows. Since the space of all compactly supported sections is dense in $\Gamma_0(\V)$ and $X$ is exhaustible by compact sets, this implies that $\Gamma_0(\V)$ is countably generated. 
\end{proof}

\subsection{Partial actions}

    The $C^*$-correspondences of interest to us are over commutative $C^*$-algebras. They arise from partial automorphisms on a compact Hausdorff space $X$ that are twisted by vector bundles. For this, we need the notion of a partial action. Details can be found in Exel's book \cite{Exel:2017}. Exel was also the first one to consider partial actions on $C^*$-algebras \cite{Exel:1994}.

\begin{definition}\label{def:partial action}
    Let $G$ be a discrete group and $X$ be a topological space. A \emph{partial action} of $G$ on $X$ consists of a collection $\{D_g\}_{g\in G}$ of open subsets of $X$ together with homeomorphisms $\theta_g:D_{g^{-1}}\to D_g$ such that 
    \begin{enumerate}
        \item we have $D_e=X$ and $\theta_e=\id_X$, where $e$ is the neutral element in the group, and 
        \item if $x\in D_{h^{-1}}$ and $\theta_h(x)\in D_{g^{-1}}$, then $x\in D_{(gh)^{-1}}$ and $\theta_{gh}(x)=\theta_g(\theta_h(x))$.
    \end{enumerate}
\end{definition}
\begin{remark}\label{rem:domains partial action}
    If $(\{D_g\}_{g\in G},\{\theta_g\}_{g\in G})$ is a partial action, then we have \[\theta_g(D_{g^{-1}}\cap D_h)=D_g\cap D_{gh},\]
    for all $g,h\in G$.
\end{remark}

If $D_g=X$ for all $g\in G$, then we recover the usual notion of an action of a group. In this case, we will somtimes say that the action is \emph{global}.

\begin{example}\label{ex:partial automorphism}
    Let $U$ and $V$ be two open subsets of $X$, and let $\theta:U\to X$ be a homeomorphism. Analogous to the global case, we obtain a partial $\Z$-action from $\theta$ by taking its powers: We define $D_{-1}\coloneqq U$ and $D_1\coloneqq V$, and then recursively $D_n\coloneqq\theta^{n-1}(D_{1-n}\cap D_1)$ and $D_{-n}\coloneqq\theta^{1-n}(D_{n-1}\cap D_{-1})$ for $n\geq 1$. One can show that $\theta_n\coloneqq\theta^n$ indeed maps from $D_{-n}$ to $D_n$, and that this defines a partial action on $X$. We also call $\theta$ a \emph{partial automorphism}, following \cite{Geffen2021}. It easily follows from the definitions of the domains that
    \[
    ...\subset D_{-2}\subset D_{-1}\subset X\supset D_1\supset D_2\supset ....
    \]

\end{example}
We define basic properties of partial actions which we will frequently encounter in the rest of this paper.
\begin{definition}
    A partial action $\{\theta_g,D_g\}_{g\in G}$ is called \emph{free} if for every $x\in D_{g^{-1}}$, the equality $\theta_g(x)=x$ implies $g=e$.
    
    A set $Y\subset X$ is called \emph{$\theta$-invariant} if for all $g\in G$ we have $\theta_g(Y\cap D_{g^{-1}})\subset Y$. A partial action is called \emph{minimal} if it does not have any closed (equivalently, open) $\theta$-invariant subsets. 
\end{definition}
A set $Y$ is invariant under a partial automorphism $\theta:U\to V$ as in example \ref{ex:partial automorphism} if and only if we have $\theta(Y\cap U)\subset Y$ and $\theta^{-1}(Y\cap V)\subset Y$.

 \subsection{Partial automorphisms, vector bundles and their \texorpdfstring{$C^*$}{C*}-correspondences}
Let $X$ be a locally compact second countable Hausdorff space, and let $\theta:U\to V$ be a partial automorphism on $X$. Let $\V$ be a vector bundle over $U$. Consider $\Gamma_0(\V)$, the space of continuous sections of $\V$ vanishing at infinity with respect to any metric on $\V$. We have shown in Section \ref{sec:Hilbert Czero modules} that $\Gamma_0(\V)$ is a Hilbert $C_0(U)$-module with inner product $\langle\xi,\eta\rangle=\beta(\xi(x),\eta(x))$, where $\beta$ is any hermitian metric on $\V$. The inner product is independent of the choice of $\beta$, up to isometry.

 The aim of this section is to construct a $C^*$-correspondence over $C_0(X)$ from the partial automorphism $\theta$ and from $\Gamma_0(\V)$. It is therefore an issue that $\Gamma_0(\V)$ is a priori only a $C_0(U)$- and not a $C_0(X)$-module. We can fix this by extending the right multiplication on $\Gamma_0(\V)$ to $C_0(X)$ by setting $(\xi f)(x)=0$ for $x\in X\backslash U$. We still denote the resulting Hilbert $C_0(X)$-module by $\Gamma_0(\V)$. It is not full since by Remark \ref{rem:sigma fp module is full} we have $\langle\Gamma_0(\V),\Gamma_0(\V)\rangle=C_0(U)$. 

 For every $f\in C_0(X)$ we define a map $\varphi(f)$ on $\Gamma_0(\V)$ by setting $(\varphi(f)\xi)(x)=\xi(x) f(\theta(x))$ for $x\in U$ and $(\varphi(f)\xi)(x)=0$ for $x\in X\backslash U$. Since $f\circ\theta$ interpreted as a function on $U$ is bounded, the expression $\xi(f\circ\theta)$ is a continuous section that lies in $\Gamma_0(\V)$. We will often write $\varphi(f)\xi=\xi(f\circ\theta)$, where it is understood that the expression $f\circ\theta$ is not well defined on its own. 
 
 One can show in the same way as in \cite{adamo2023} that $\varphi$ maps into the adjointable operators. Hence $\Gamma_0(\V)$ together with $\varphi$ forms a $C^*$-correspondence over $C_0(X)$, which we will denote by $\Gamma_0(\V,\theta)$. If it is clear what the vector bundle and the partial automorphism are, then we will often write $\E$ for $\Gamma_0(\V,\theta)$. In that case the structure map is denoted by $\varphi_\E$ and the inner product by $\langle\cdot,\cdot\rangle_\E$.
 \begin{lemma}\label{lem:katsura ideal}
We have
\begin{align*}
    \ker\varphi_\E=\{f\in C_0(X):f|_V=0\},\quad (\ker\varphi_\E)^\perp=C_0(\interior(\overline{V})), \quad\text{and}\quad \varphi_\E^{-1}(\mathcal{K}(\E))=C_0(V).\quad 
\end{align*}
In particular
\[ J_\E:= \varphi_\E^{-1}(\mathcal{K}(\E))\cap(\ker\varphi_\E)^\perp=C_0(V).\]
\end{lemma}
\begin{proof}
 For every function $f$ in $C_0(X)$ such that $f|_V=0$ we have
\begin{align*}
    \varphi_\E(f)\xi=\xi(f\circ\theta)=0
\end{align*}
for all $\xi\in\E$, and hence 
\begin{align*}
    \{f\in C_0(X):f|_V=0\}\subset\ker \varphi_\E.
\end{align*}
On the other hand, let $f$ be in $C_0(X)$ such that $f(x_0)\neq0$ for some $x_0\in V$. Then there exists an open set $W\subset V$ such that $f(x)\neq0$ for every $x\in W$. Shrinking $W$ if necessary we may assume there exists an isomorphism $h:\theta^{-1}(W)\times\C^n\to\V|_{\theta^{-1}(W)}$ for some $n\in\N$. Let $g\in C_0(\theta^{-1}(W))\subset C_0(U)$ be nonzero, and let $v$ be any nonzero vector in $\C^n$. Then $h(\cdot,g(\cdot)v)$ lies in $\E$, and there is $x\in\theta^{-1}(W)$ such that $\varphi_\E(f)h(x,g(x)v)=h(x,g(x)v)f\circ\theta(x)\neq0$. Hence 
\begin{align*}
    \ker \varphi_\E=\{f\in C_0(X):f|_V=0\}=\{f\in C_0(X):f|_{\overline{V}}=0\}=C_0(X\backslash\overline{V}),
\end{align*}
which implies 
\begin{align}\label{eq:prop:bimodule2}
    (\ker\varphi_\E)^\perp=C_0(X\backslash\overline{V})^\perp=C_0(\interior(\overline{V})).
\end{align}
That $\varphi_\E^{-1}(\mathcal{K}(\E))$ is equal to $C_0(V)$ follows from Proposition \ref{cor:compacts for sigma fgp module}.
\end{proof}
\begin{lemma}\label{lem:line bundle bimodule}
    If $\V$ is a line bundle, meaning that all fibers of $\V$ have dimension one, then $\xi\langle\eta,\zeta\rangle_\E=\zeta\langle\eta,\xi\rangle_\E$ for all $\xi,\eta,\zeta\in\E$. In particular, $\E$ is a Hilbert $C_0(X)$-bimodule.
\end{lemma}
\begin{proof}
    The proof is basically the same as the proof of \cite[Proposition 3.5]{adamo2023}. Let $\{U_i,h_i\}_{i\in I}$ be an atlas for $\V$ with transition functions $g_{ij}$. Choose $i,j$ and $k$ in $I$ such that $U_i\cap U_j\cap U_k\neq\emptyset$. Define sections $\xi(x)\coloneqq h_i(x,r(x))$, $\eta(x)=h_j(x,s(x))$ and $\zeta(x)=h_k(x,t(x))$ for continuous functions $r\in C_0(U_i)$, $s\in C_0(U_j)$ and $t\in C_0(U_k)$. Take $x\in U_i\cap U_j\cap U_k$. We have 
    \[\langle\eta,\xi\rangle_\E(x)=\langle h_j^{-1}(\eta(x)),h_j^{-1}(\xi(x))\rangle_\C=\langle s(x),g_{ji}(x)r(x)\rangle_\C=\overline{s(x)}g_{ji}r(x)\]
    and hence
    \[\zeta\langle\eta,\xi\rangle_\E(x)=h_k(x,\overline{s(x)}g_{ji}r(x)t(x)).\]
    Swapping the roles of $\xi$ and $\zeta$ we obtain
    \begin{align*}
        \xi\langle\eta,\zeta\rangle_\E(x)&=h_i(x,\overline{s(x)}g_{jk}r(x)t(x))=h_k(x,g_{ki}\overline{s(x)}g_{jk}r(x)t(x))\\&=h_k(x,\overline{s(x)}g_{ji}r(x)t(x))=\zeta\langle\eta,\xi\rangle_\E(x). 
    \end{align*}
   By Lemma \ref{lem:frame vector bundle} and Lemma \ref{lem:inductive limit exhaustion by compact sets} it is enough to consider such sections.
\end{proof}
\begin{proposition}\label{prop:line bundle bimodule}
    We have $\varphi_\E(C_0(V))=\K(\E)$ if and only if $\V$ is a line bundle. In particular $\E$ is a $C_0(X)$-Hilbert bimodule with left inner product 
    \[{}_\E\langle\xi,\eta\rangle=\langle\eta,\xi\rangle_\E\circ\theta^{-1},\quad\text{for all }\xi,\eta\in\E,\] if and only if $\V$ is a line bundle.
\end{proposition}
\begin{proof}
Assume that there exists $x\in X$ such that $p^{-1}(x)$ has dimension greater than one. There exists a chart $(W,h)$ such that $x$ lies in $W$, where $h:W\times\C^n\to p^{-1}(W)$ with $n>1$. Take a matrix $T\in M_n(\C)$ which is not a multiple of the identity. Let $\Tilde{W}$ be an open neighborhood of $x$ contained in $W$, and let $\chi$ be a continuous function which is equal to one on $\Tilde{W}$ and vanishes outside of $W$. Let $t$ be the map on $\E$ that maps a section $\xi$ to the section $(t\xi)(y)=h(y,Th^{-1}(\xi(y)))$. One can show that $t$ is well defined and an element of $\calL(\E)$. Proposition \ref{cor:compacts for sigma fgp module} shows that $t$ lies in $\K(\E)$. However $t$ clearly does not lie in $\varphi_\E(C_0(V))$.

Now we show the other direction. By Lemma \ref{lem:katsura ideal}, Katsura's ideal $J_\E$ is equal to $C_0(V)$, and hence $\varphi_\E(C_0(V))$ is contained in $\K(\E)$. It is left to show that the restriction of $\varphi_\E$ to $J_\E$ is surjective onto $\K(\E)$. If $\V$ is a line bundle then we have $\xi\langle\eta,\zeta\rangle_\E=\zeta\langle\eta,\xi\rangle_\E$ for all $\xi,\eta,\zeta\in\E$ by Lemma \ref{lem:line bundle bimodule}. This implies 
\[\varphi_\E(\langle\eta,\xi\rangle_\E\circ\theta^{-1})\zeta=\zeta\langle\eta,\xi\rangle_\E=\xi\langle\eta,\zeta\rangle_\E=\theta_{\xi,\eta}\zeta\] and hence $\varphi_\E^{-1}(\theta_{\xi,\eta})=\langle\eta,\xi\rangle_\E\circ\theta^{-1}$. Since the rank-one operators densely span $\K(\E)$ we obtain $\varphi_\E(C_0(V))=\K(\E)$. The rest of the statement follows from the discussion following Definition~\ref{def:bimodule}.
\end{proof}
\begin{example}
    Let $\alpha$ be a homeomorphism. Take the trivial line bundle over $X$, namely $\V=(X\times\C,p,X)$ where $p$ is the projection onto the second factor. Then $\Gamma_0(\V,\alpha)$ is exactly the $C^*$-correspondence from Example \ref{ex:homomorphisms as correspondences}, with $A=B=C_0(X)$ and $\phi$ the $\ast$-homomorphism on $C_0(X)$ induced by $\alpha$. The Cuntz--Pimsner algebra $\CP(\Gamma(\V,\alpha))$ is isomorphic to the crossed product $C_0(X)\ltimes_\alpha\Z$. It instead of $\alpha$ we take a partial automorphism $\theta:U\to V$ and consider the trivial line bundle over $U$, then the Cuntz--Pimsner algebra is isomorphic to the partial crossed product $C_0(X)\ltimes_\theta\Z$, see \cite[Proposition 2.23]{Muhly:1998}. 
    \end{example}
    \begin{example}\label{ex:cuntz algebra from Cuntz--Pimsner algebra}
        Take the space $X$ to be a single point $\{\mathrm{pt}\}$. Then the partial automorphism as well as the vector bundle are necessarily trivial. Consider the trivial vector bundle of rank $n$ over $X=\{\mathrm{pt}\}$, namely $\V=(X\times\C^n,p,X)$. Then a frame for the module $\Gamma(\V)$ is given by an orthonormal basis of $\C^n$. Under the universal covariant representation the basis vectors give rise to $n$ commuting isometries whose range projections sum up to $1$. Hence the Cuntz--Pimsner algebra is isomorphic to the Cuntz algebra $\CP_n$. 
    \end{example}
    \begin{example}\label{ex:continuous field of cuntz algebras}
        Take a locally compact second countable Hausdorff space $X$ and a vector bundle $\V$ over $X$. Consider the trivial action $\id_X$. This clearly generalizes Example \ref{ex:cuntz algebra from Cuntz--Pimsner algebra}. By \cite[Proposition 2]{Vasselli:2005} the Cuntz--Pimsner algebra $\CP(\E)$ is a continuous field of Cuntz algebras. This type of algebra was also studied in \cite{Dadarlat:2012}.
    \end{example}

    At the end of this section we prove a technical result which will be useful in the proofs of Lemmas \ref{lem:conditional expectation approximation} and \ref{lem:changing order}.

    \begin{proposition}\label{prop:tensor product fullness}
   Let $\E = \Gamma_0(\V,\theta)$ and  $\{D_{n}\}_{n \in \mathbb{Z}}$ denote the open subsets defining the partial action as in Example~\ref{ex:partial automorphism}. We have $\langle\E^{\otimes n},\E^{\otimes n}\rangle_{\E^{\otimes n}}=C_0(D_{-n})$ for all $n\in\N$. 
\end{proposition}
\begin{proof}
    We show the claim inductively. The case $n=1$ follows from Remark \ref{rem:sigma fp module is full}. Now take $n>1$ and assume that the claim holds for $n-1$. Take $\xi_1,\eta_1$ in $\E^{\otimes(n-1)}$ and $\xi_2,\eta_2$ in $\E$. Also take $f$ in $C_0(D_{1})$. To ease notation we will write $\langle\cdot,\cdot\rangle_n$ instead of $\langle\cdot,\cdot\rangle_{\E^{\otimes n}}$. We calculate
    \begin{align*}
        \langle\xi_1\otimes\xi_2,\eta_1\otimes\varphi_\E(f)\eta_2\rangle_{n}=\langle\xi_2,\varphi_\E(\langle\xi_1,\eta_1\rangle_{n-1}f)\eta_2\rangle_\E=\langle\xi_2,\eta_2\rangle_\E(\langle\xi_1,\eta_1\rangle_{n-1}f\circ\theta).
    \end{align*}
    Note that $f$ was only needed to make everything well defined, and can be introduced without loss of generality. By the induction hypothesis $\langle\xi_1,\eta_1\rangle_{n-1}$ is contained in $C_0(D_{1-n})$. Hence the open support of $\langle\xi_1,\eta_1\rangle_{n-1}f\circ\theta$ is contained in $\theta^{-1}(D_{1-n}\cap V)$, which is equal to $D_{-n}$. This shows that $\langle\E^{\otimes n},\E^{\otimes n}\rangle_n$ is contained in $C_0(D_{-n})$. The other inclusion follows from the above calculation and Remark \ref{rem:sigma fp module is full}.
\end{proof}
\subsection{Grading and the fixed point algebra}\label{sec:grading and the fixed point algebra}
The goal of the next two sections is to understand the structure of the fixed point algebra of the Cuntz--Pimsner algebra associated to a partial automorphism twisted by a vector bundle. We start out in this section by reviewing some of the general theory concerning gradings and fixed point algebras of Cuntz--Pimsner algebras. 
 
Let $G$ be a discrete abelian group. Then a $C^*$-algebra $A$ is \emph{$G$-graded} if there exists a collection of linearly independent subspaces $\{A^g\}_{g\in G}$ of $A$ such that the direct sum over all $A^g$ is dense in $A$, and such that $A^gA^h\subset A^{gh}$ as well as $(A^g)^*\subset A^{g^{-1}}$ hold for all $g\in G$, see \cite[Definition 16.2]{Exel:2017}. If there also exists a conditional expectation onto $A^1$ vanishing on $A^g$ for all $g\neq 1$, then we speak of a \emph{topological grading}, see \cite[Definition 19.2]{Exel:2017}.

Write $\widehat{G}$ for the Pontryagin dual of $G$. An action of the compact abelian group $\widehat{G}$ on a $C^*$-algebra $A$ induces a topological grading of $A$ by $G$.  For a proof of this well known result, see \cite[Proposition A4]{Sehnem:2019}.  Thus the existence of the gauge action $\gamma : \mathbb{T} \curvearrowright \CP(\E)$ on the Cuntz--Pimsner algebra $\CP(\E)$ implies that $\CP(\E)$ is $\Z$-graded. The subspace of $\CP(\E)$ associated to an integer $k$ is given by 
 \[\CP(\E)^k\coloneqq\{a\in\CP(\E):\gamma_z(a)=z^ka\text{ for all }z\in\mathbb{T}\}.\]   

 A special role is played by the subspace $\CP(\E)^0$ associated to $0\in\Z$. It is  a closed subalgebra of $\CP(\E)$, which can easily be seen from the definition. We call $\CP(\E)^0$ the \emph{fixed point algebra} of the gauge action. Since the grading on the Cuntz--Pimsner algebra is topological, there exists a conditional expectation $\Phi$ from $\CP(\E)$ onto the fixed point algebra. This conditional expectation is faithful \cite[Section 4]{katsura:2004_2}. 
 
 \begin{definition}[{\cite[Definition 2.5]{Katsura:2004}}]\label{def:rep of tensor product}
Let $(\pi,t)$ be a representation of $\E$. We set $t^0=\pi$ and $t^1=t$. For $n=2,3,...$ we define a linear map $t^n:\E^{\otimes n}\to C^*(\pi,t)$ by $t^n(\xi\otimes\eta)=t(\xi)t^{n-1}(\eta)$ for $\xi\in\E$ and $\eta\in\E^{\otimes(n-1)}$. 
\end{definition}
For every $n\in\N$ the map $t^n$ is well defined and $(\pi,t^n)$ is a representation of $\E^{\otimes n}$. We obtain $\ast$-homomorphisms $\psi_{t^n}$ from $\K(\E^{\otimes n})$ to $C^*(\pi,t)$. On rank-one operators $\psi_{t^n}$ is given by $\psi_{t^n}(\theta_{\xi,\eta})=t^n(\xi)t^n(\eta)^*$ for $\xi$ and $\eta$ in $\E^{\otimes n}$. If $(\pi,t)$ is injective then $t^n$ and $\psi_{t^n}$ are isometric. This is in particular the case when $(\pi,t)$ is the universal representation $(\pi_\E,t_\E)$. To ease notation we write $\psi_n$ for $\psi_{t_\E^n}$. 

For $n\in\N$ we set 
\[
B_n\coloneqq\psi_n(\K(\E^{\otimes n})).
\] Then $B_n$ is a subalgebra of $\CP(\E)$. Define the subalgebra $B_{[m,n]}\subset\CP(\E)$ for $m,n\in\N$ with $m\leq n$ by 
\[
B_{[m,n]}=B_m+B_{m+1}+...+B_n.
\]
Since $\psi_n$ is isometric, it is clear that $B_n$ is isomorphic to $\K(\E^{\otimes n})$ as $C^*$ algebras. Furthermore $B_{[k,n]}$ is an ideal of $B_{[m,n]}$ whenever $m\leq k\leq n$, see \cite{Katsura:2004}.
The inductive limit of the increasing sequence $\{B_{[m,n]}\}_{n=m}^\infty$ is denoted by $B_{[m,\infty)}$. The $C^*$-algebra 
\[
B_{[0,\infty)} = \lim_{n \to \infty} B_{[0,n]}
\]
is called the \emph{core} of $\CP(\E)$. 

\begin{proposition}[{\cite[Proposition 5.7]{Katsura:2004}}]\label{prop:core and fixed point algebra}
The core $B_{[0,\infty)}$ coincides with $\CP(\E)^0$, the fixed point algebra of the gauge action. 
\end{proposition}
It follows from the definition of the gauge action that $t^n_\E(\xi)t^m_\E(\eta)^*$ lies in $\CP(\E)^k$ if and only if $n-m=k$, for every $\xi\in\E^{\otimes n}$ and $\eta\in\E^{\otimes m}$. The linear span of such elements is dense in $\CP(\E)^k$; this can be shown in the same way as \cite[Proposition 5.7]{Katsura:2004}.

If $\E$ is a Hilbert $A$-bimodule, then 
\[
    t_\E(\xi)t_\E(\eta)^*=\psi_{t_\E}(\theta_{\xi,\eta})=\psi_{t_\E}(\varphi_\E({}_\E\langle\xi,\eta\rangle))=\pi_\E({}_\E\langle\xi,\eta\rangle)
\]
for all $\xi,\eta\in\E$. In particular the fixed point algebra is $\ast$-isomorphic to $A$, and $\CP(\E)^n$ and $\E^{\otimes n}$ are isomorphic as Hilbert $A$-bimodules, for all $n\in\Z\backslash\{0\}$.

We define the ideal $I_\E$ of $A$ to be  
\begin{equation} \label{eq:IE} I_\E\coloneqq\{a\in A:\pi_\E(a)\in \psi_t(\K(\E))\}.
\end{equation}
For every $n\in\N$ we define the subalgebra $B_n'$ of $B_n$ to be 
\begin{equation} \label{eq:Bn'} B_n'\coloneqq\psi_n(\K(\E^{\otimes n}I_\E)).
\end{equation}

\begin{proposition}[{\cite[Propositions 5.9 and 5.12]{Katsura:2004}}]\label{prop:commutative diagram for cores}
For every $n\in\N$ we have $B_n\cap B_{n+1}=B_n'$ and $B_{[0,n]}\cap B_{n+1}=B_n'$. The diagram 
\[\begin{tikzcd}
	0 & {B_n'} & {B_{[0,n]}} & {B_{[0,n]}/B_n'} & 0 \\
	0 & {B_{n+1}} & {B_{[0,n+1]}} & {B_{[0,n]}/B_n'} & 0
	\arrow[from=1-1, to=1-2]
	\arrow[from=1-2, to=1-3]
	\arrow[from=1-3, to=1-4]
	\arrow[from=1-4, to=1-5]
	\arrow[from=2-1, to=2-2]
	\arrow[from=2-2, to=2-3]
	\arrow[from=2-3, to=2-4]
	\arrow[from=2-4, to=2-5]
	\arrow[no head, from=1-4, to=2-4]
	\arrow[shift right, no head, from=1-4, to=2-4]
	\arrow[from=1-3, to=2-3]
	\arrow[from=1-2, to=2-2]
\end{tikzcd}\]
is commutative with exact rows, where all the maps are the obvious inclusion or quotient maps. 
\end{proposition}
\subsection{Fibered structure of the fixed point algebra}
In this section we investigate the fixed point algebra of the Cuntz--Pimsner algebra associated to a partial automorphism twisted by a vector bundle. We show that it is a $C_0(X)$-algebra, where $X$ is the base space of the vector bundle. We calculate the fibers. 

We write $\E$ for the Hilbert $C_0(X)$-module $\Gamma_0(\V,\theta)$, where $\theta:U\to V$ is a partial automorphism and $\V$ is a vector bundle over $U$. Section \ref{sec:Hilbert Czero modules} shows that both $\E$ and $\K(\E)$ naturally fiber over $U$, and hence over $X$: More precisely, the fiber of $\E$ at a point $x\in U$ is isomorphic to $\C^{d(x)}$, where $d(x)$ is the rank of the vector bundle at the point $x$. For $x\in X\backslash U$ the fiber is $\{0\}$, the trivial Hilbert space. The fiber of $\K(\E)$ at a point $x\in U$ is $\K(\E_x)\cong M_{d(x)}(\C)$, and at a point $x\in X\backslash U$ it is the trivial $C^*$-algebra, which we also denote by $\{0\}$. 

There is another $C_0(X)$-structure on $\E$ and $\K(\E)$ coming from the structure map, which will be more useful to us. To distinguish it from the $C_0(X)$-structure coming from right multiplication, we will denote the fiber of $\E$ over a point $x\in X$ by $\E_{\varphi,x}$. This fiber is equal to the quotient of $\E$ by the submodule $\varphi_\E(I_x)\E$, where we write $I_x$ for $C_0(X\backslash\{x\})$. Similarly we denote the fiber of $\K(\E)$ over a point $x\in X$ by $\K(\E)_{\varphi,x}$.

We can also consider the $n$-fold internal tensor product $\E^{\otimes n}$ and the compact operators $\K(\E^{\otimes n})$. On them the structure map $\varphi_{\E^{\otimes n}}$ induces a $C_0(X)$-structure as well. The fibers over a point $x\in D_n$ are given by 
\[(\E^{\otimes n})_{\varphi,x}\cong\E_{\theta^{-1}(x)}\otimes ... \otimes \E_{\theta^{-n}(x)},\]
and
\[ \K(\E^{\otimes n})_{\varphi,x}\cong M_{d(\theta^{-1}(x))}\otimes... M_{d(\theta^{-n}(x))},\]
where $d(\theta^{-j}(x))$ denotes the rank of the vector bundle at $d(\theta^{-j}(x))$. For $x\in X\backslash D_n$ the fibers are given by the trivial Hilbert space and the trivial $C^*$-algebra, respectively. 

    If $k\in\K(\E^{\otimes n})$ is a compact operator such that $k\otimes \id_\E\in\K(\E^{\otimes(n+1)})$, then $(k\otimes \id_\E)_{\varphi,x}$ is equal to $k_{x}\otimes \id_{\E_{\varphi,x}}$. This follows from the definition of the tensor product of vector spaces, and the fact that $(\id_\E)_{\varphi,x}$ equals $\id_{\E_{\varphi,x}}$. 

We now investigate the $C_0(X)$-structure of the fixed point algebra.

\begin{proposition}\label{prop:fixed point algebra is C(X) algebra}
Let $(\pi_\E, t_\E)$ be the universal covariant representation of $\E$. Then $\pi_\E$ maps $C_0(X)$ into the center of $\CP(\E)^0$. This makes $\CP(\E)^0$ a $C_0(X)$-algebra with structure map $\pi_\E$, and induces a $C_0(X)$-structure on the subalgebras $B_n$ and $B_{[0,n]}$.
 \end{proposition}
\begin{proof} 
  To ease notation we write $\pi$ instead of $\pi_\E$, and $t$ instead of $t_\E$. It is clear that $\pi$ maps into the fixed point algebra. By definition of the structure map $\varphi_\E$, we have
\begin{align*}
    \pi(f)t^n(\xi)t^n(\eta)^*&=t^n(\varphi_{\E^{\otimes n}}(f)\xi)t^n(\eta)^*=t^n(\xi)t^n(\eta(f\circ\theta^n))^*=t^n(\xi)t^n(\varphi_{\E^{\otimes n}}(f)\eta)^*\\&=t^n(\xi)t^n(\eta)^*\pi(f)
\end{align*}
for all $\xi,\eta\in\E^{\otimes n}$ and all $f\in C_0(X)$. Note that the expression $f\circ\theta^n$ is not well defined on its own, but makes sense if multiplied from the right to an element $\eta\in\E^{\otimes n}$. For every $f$ in $C_0(X)$ and all $n\in\N$ we obtain that $\pi(f)$ commutes with all elements of $B_n$, and hence with all elements of $B_{[0,n]}$. By Proposition \ref{prop:core and fixed point algebra}, the fixed point algebra is the inductive limit of the $B_{[0,n]}$, and so we obtain that $\pi$ maps into the center of the fixed point algebra. 

If $(u_i)_{i\in I}$ is an approximate unit for $C_0(X)$ then $\xi u_i$ converges to $\xi$ for all $\xi\in\E$. We obtain that $\pi(u_i)$ is an approximate unit for $\CP(\E)^0$, and hence $\pi(C_0(X))\CP(\E)^0$ is dense in $\CP(\E)^0$. This shows that $\pi$ indeed turns the fixed point algebra into a $C_0(X)$-algebra.
\end{proof}
The situation is now the following: For every $n\in\N$ we have an isomorphism $\psi_n$ from $\K(\E^{\otimes n})$ to $B_n\subset\CP(\E)^0$. The structure map $\varphi_{\E^{\otimes n}}$ induces a $C_0(X)$-structure on $\K(\E^{\otimes n})$. On the other hand Proposition \ref{prop:fixed point algebra is C(X) algebra} shows that $\pi_\E$ induces a $C_0(X)$-structure on $B_n$. It is straightforward to see that $\psi_n$ is $C_0(X)$-linear with respect to these structures. This means that it is an isomorphism of $C_0(X)$-algebras, and in particular it induces isomorphisms $\psi_{n,x}$ on the fibers for every $x\in X$. Since we have already computed the fibers of $\K(\E^{\otimes n})$, this yields the fibers of the $B_n$. Most of the work will then consist in calculating the fibers of the subalgebras $B_{[0,n]}$. Once this is done,  we use that $\CP(\E)^0$ can be obtained as an inductive limit of the $B_{[0,n]}$. 

\begin{lemma}\label{lem:calculating ideal for core}
For the ideal $I_\E$ and the subalgebra $B_n'\subset B_n$ of \eqref{eq:IE} and \eqref{eq:Bn'} respectively, we have $I_\E=C_0(V)$ and $B_n'=\pi_\E(C_0(D_{n+1}))B_n$.
\end{lemma}
\begin{proof}
Using Proposition \ref{cor:compacts for sigma fgp module} and the fact that the fibers of $\K(\E)$ with respect to the $C_0(X)$-structure induced by $\varphi_\E$ are shifted by $\theta$ compared to the fibers with respect to the $C_0(X)$-structure induced by right multiplication, we obtain
\begin{equation}\label{eq:compacts for fibered structure induced by structure map}\K(\E)=\{t\in\calL(\E):(x\mapsto\norm{t_x})\in C_0(V)\}.\end{equation}
Now take $f\in C_0(X)$ and assume that there exists a compact operator $k$ such that $\pi_\E(f)=\psi_{t_\E}(k)$. By definition of the $C_0(X)$-structure on $B_n$ we have $\pi(f)_x=f(x)$ for all $x\in X$. Since $\psi_{t_\E}$ is a $C_0(X)$-linear isomorphism this and (\ref{eq:compacts for fibered structure induced by structure map}) imply $f\in C_0(V)$. This shows that $I_\E$ is contained in $C_0(V)$.

On the other hand, by Lemma \ref{lem:katsura ideal} Katsura's ideal $J_\E$ is equal to $C_0(V)$. By covariance of the universal representation $J_\E$ is contained in $I_\E$. This establishes the first equality of the lemma,  $I_\E=C_0(V)$. 

To show the remaining part of the lemma, first note 
\[\E^{\otimes n}C_0(V)=\varphi_{\E^{\otimes n}}(C_0(\theta^n(D_{-n}\cap V)))\E^{\otimes n}.\]
Since $\theta^n(D_{-n}\cap V)$ equals $D_{n+1}$ we obtain
\begin{align*}
B_n'&=\psi_n(\K(\E^{\otimes n}I_\E))=\psi_n(\varphi_{\E^{\otimes n}}(C_0(D_{n+1}))\K(\E^{\otimes n}))\\&=\pi_\E(C_0(D_{n+1}))\psi_n(\K(\E^{\otimes n}))=\pi_\E(C_0(D_{n+1}))B_n
\end{align*}
which concludes the proof.
\end{proof}
By Proposition \ref{prop:commutative diagram for cores} for any Cuntz--Pimsner algebra, $B_n'$ is contained in $B_{n+1}$ (Lemma \ref{lem:calculating ideal for core} also establishes this for the specific case $\E = \Gamma_0(\mathcal{V}, \alpha)$.). This induces an inclusion map from $\psi_n^{-1}(B_n')$ to $\K(\E^{\otimes n})$ which we compute in the next lemma.

\begin{lemma}\label{lem:inclusion map cores compacts}
Let $n \in \mathbb{N}$ and denote by $\iota^n$ the inclusion map from $B_n'$ into $B_{n+1}$. Then we have \[\psi_n^{-1}\circ\iota^n\circ\psi_n(k)=k\otimes \id_\E\]
    for all $k\in\psi_n^{-1}(B_n')$, and any $n \in \mathbb{N}$.
\end{lemma}

\begin{proof}
    To ease notation we will drop the subscripts and write $\pi$ and $\varphi$ instead of $\pi_\E$ and $\varphi_{\E^{\otimes n}}$, respectively. We will drop the $t^n$ entirely so an expression of the form $\xi\eta^*$ for $\xi,\eta\in\E^{\otimes n}$ is understood to mean $t^n(\xi)t^n(\eta)^*$. 
    
    From Lemma \ref{lem:frame for submodule}, $\E$ is countably generated, and hence has a countable frame $(\zeta_i)_{i\in \N}$. 
    
    Take $g\in I_\E$. By the definition of $I_\E$ there exists $k\in\K(\E)$ such that $\pi(g)=\psi_{t}(k)$. Using Lemma \ref{lem:frame induces approx unit in compacts} we obtain 
    \begin{equation}\label{eq:inclusion map compacts}\sum_{i\in\N}\zeta_i\zeta_i^*\pi(g)=\sum_{i\in\N}\psi_{t}(\theta_{\zeta_i,\zeta_i})\psi_{t}(k)=\psi_{t}(\sum_{i\in\N}\theta_{\zeta_i,\zeta_i}k)=\psi_{t}(k)=\pi(g)\end{equation}
    where the series converges in norm.

    We now prove the statement of the Lemma. By Lemma \ref{lem:calculating ideal for core} we have \[\psi_n^{-1}(B_n')=\varphi(C_0(D_{n+1}))\K(\E^{\otimes n}).\] It therefore suffices to prove the statement for compact operators of the form $\varphi(f)\theta_{\xi,\eta}$ with $f\in C_0(D_{n+1})$ and $\xi,\eta\in\E^{\otimes n}$. The image of such an operator under $\psi_n$ is $\pi(f)\xi\eta^*$. Using Equation (\ref{eq:inclusion map compacts}) we calculate
    \[
    \pi(f)\xi\eta^*=\xi\pi(f\circ\theta^n)\eta^*=\sum_{i\in\N} \xi\zeta_i\zeta_i^*\pi(f\circ\theta^n)\eta^*=\sum_{i\in\N}\pi(f) \xi\zeta_i\zeta_i^*\eta^*,
    \]
    and hence 
    \begin{align*}
        \psi_n^{-1}\circ\iota^n\circ\psi_n(\varphi(f)\theta_{\xi,\eta})&=\psi_n^{-1}(\sum_{i\in\N}\pi(f) \xi\zeta_i\zeta_i^*\eta^*)=\sum_{i\in\N}\psi_n^{-1}(\pi(f) \xi\zeta_i\zeta_i^*\eta^*)\\&=\sum_{i\in\N}\varphi(f)\theta_{\xi\otimes\zeta_i,\eta\otimes\zeta_i}=\varphi(f)\theta_{\xi,\eta}\otimes\id_\E,
    \end{align*}
    where the last equality follows from the fact that $(\zeta_i)_{i\in\N}$ is a frame. This concludes the proof.
\end{proof}

Before we can prove the main theorem of this section we make a few observations.

A diagram 
\[
\xymatrix{
  A \ar[r]^{\psi} \ar[d]_{\varphi} & C \ar[d]^{\gamma} \\
  B \ar[r]_{\beta} & D
}
\]
is \emph{pushout} of $C^*$-algebras is a pushout if $D$ is generated by $\beta(B)\cup \gamma(C)$ and if every other pair of morphisms $\alpha : B \to E$ and $ \delta : C \to E$ into a $C^*$-algebra $E$ satisfying $\alpha \circ \varphi = \delta \circ \psi$, there is a (necessarily unique) $*$-homomorphism $\sigma : D \to E$ satisfying $\alpha = \sigma \circ  \beta$ and $\delta = \sigma \circ \gamma$ \cite[Section 2.3]{pedersen:1999}

\begin{lemma}\label{pushout C(X)-algebras}
    Let $A$, $B$, $C$ and $D$ be $C_0(X)$-algebras for some locally compact Hausdorff space $X$. Consider the commutative diagram 
\[\begin{tikzcd}
	A & B \\
	C & D
	\arrow["\beta", from=1-1, to=1-2]
	\arrow["\alpha", from=1-1, to=2-1]
	\arrow["\delta", from=1-2, to=2-2]
	\arrow["\gamma", from=2-1, to=2-2]
\end{tikzcd}\]
where all the morphisms are $C_0(X)$-linear and injective. Assume that the diagram is a pushout in the sense of \cite{pedersen:1999}. In addition, assume that $\beta(A)$ is a hereditary subalgebra in $B$, and that $\alpha(A)C=C$. Then the diagram
\[\begin{tikzcd}
	{A_x} & {B_x} \\
	{C_x} & {D_x}
	\arrow["{\alpha_x}", from=1-1, to=2-1]
	\arrow["{\delta_x}", from=1-2, to=2-2]
	\arrow["{\beta_x}", from=1-1, to=1-2]
	\arrow["{\gamma_x}", from=2-1, to=2-2]
\end{tikzcd}\]
is a pushout for every $x\in X$.
\end{lemma}
\begin{proof}
    The statement follows from \cite[Proposition 4.5]{pedersen:1999} together with \cite[Theorem 9.3]{pedersen:1999}.
\end{proof}
\begin{remark}\label{rem:inductive system of C(X)-algebras}
    It is not hard to show that if a $C_0(X)$-algebra $A$ is the limit of an inductive system $(A_n,\varphi)$, where the $A_n$ are $C_0(X)$-algebras and the $\varphi_n$ are injective and $C_0(X)$-linear, then $A_x$ is the limit of the inductive system $((A_n)_x,(\varphi_n)_x)$. 
\end{remark}
\begin{lemma}\label{lem:continuous field of Cstar algebras}
    Let $X$ be a locally compact metrizable space, and let $A$ be a separable $C_0(X)$-algebra all of whose fibers are simple and nonzero. If, for every compact subset $K$ of $X$, there exists $e_K\in A$ such that $\norm{e_K(x)}\geq 1$ for all $x$ in $K$, then $A$ is continuous. 
\end{lemma}
\begin{proof}
    Let $x_0$ be in $X$, and let $W$ be an open neighborhood of $x_0$ with compact closure. By \cite[Lemma 2.3]{Dadarlat:2009} the function $x\mapsto\norm{q_{\overline{W}}(a)(x)}$ is continuous in $x_0$. Since $q_{\overline{W}}$ is an isomorphism on each fiber and hence preserves the norm, this proves the claim. 
\end{proof}
\begin{theorem}\label{thm: fixed point algebra}
    Let $\E=\Gamma_0(\V,\theta)$ be the $C^*$-correspondence associated to a partial automorphism $\theta:U\to V$ on a locally compact second countable Hausdorff space $X$, and a vector bundle $\V$ over $U$. Then the fixed point algebra $\CP(\E)^{0}\subset\CP(\E)$ is a continuous $C_0(X)$-algebra. 
    
    The fiber at a point $x\in\bigcap_{k>0}D_k$ is the UHF algebra given by the limit of the inductive system $(A_n,\varphi_n)$, where \[A_n=M_{d(\theta^{-1}(x))}\otimes...\otimes M_{d(\theta^{-n}(x))}\quad\text{ and }\quad\varphi_n(a)=a\otimes I_{d(\theta^{-n-1}(x))}\]
    with the identity matrix $I_{d(\theta^{-n-1}(x))}\in M_{d(\theta^{-n-1}(x))}$.  

    If $x$ lies in $D_n$ but not $D_{n+1}$ for some $n\in\N$, then the fiber at $x$ is $M_{d(\theta^{-1}(x))}\otimes...\otimes M_{d(\theta^{-n}(x))}$. If $x$ does not lie in $D_1=V$, then the fiber at $x$ is $\C$.
\end{theorem}
\begin{proof}
     Proposition \ref{prop:fixed point algebra is C(X) algebra} shows that $\CP(\E)^0$ is a $C_0(X)$-algebra. It remains to compute the fibers. 

    From Proposition \ref{prop:core and fixed point algebra} we know that $\CP(\E)^0$ is the limit of the inductive system $(B_{[0,n]},\phi^n)$ where $\phi_n$ is simply the inclusion map. Therefore we need to calculate the fibers of $B_{[0,n]}$ and determine the maps $\phi^n_x$ induced on the fibers.

      Recall the commutative diagram
\[\begin{tikzcd}
	{B_n'} & {B_{[0,n]}} \\
	{B_{n+1}} & {B_{[0,n+1]}}
	\arrow["{\rho^n}", from=1-1, to=1-2]
	\arrow[from=2-1, to=2-2]
	\arrow["{\phi^n}"', from=1-2, to=2-2]
	\arrow["{\iota^n}"', from=1-1, to=2-1]
\end{tikzcd}\]
from Proposition \ref{prop:commutative diagram for cores}, where we have given names to the maps now. Lemmas \ref{lem:calculating ideal for core} and \ref{lem:inclusion map cores compacts} together with \cite[Lemma 5.10]{Katsura:2004} imply that $\iota^n(B_n')B_{n+1}$ equals $B_{n+1}$. Theorem 2.4 in \cite{pedersen:1999} then shows that the diagram  is a pushout. From Lemma \ref{pushout C(X)-algebras} we obtain that the diagram 
\begin{equation}\label{eq:diagram pushout fibers}\begin{tikzcd}
	{(B_n')_x} & {(B_{[0,n]})_x} \\
	{(B_{n+1})_x} & {(B_{[0,n+1]})_x}
	\arrow["{\iota^n_x}"', from=1-1, to=2-1]
	\arrow["{\rho_x^n}", from=1-1, to=1-2]
	\arrow["{\phi^n_x}", from=1-2, to=2-2]
	\arrow[from=2-1, to=2-2]
\end{tikzcd}\end{equation}
is a pushout. 

For $n\in\N_0$ and $x\in D_n$ write 
\[M_{n,x}\coloneqq M_{d(\theta^{-1}(x))}\otimes...\otimes M_{d(\theta^{-n}(x))}\]
where it is understood that $M_{0, x}=0$ for all $x\in X$. 

We will now show inductively that for all $n\in\N_0$ and all $x\in D_n$ the fiber $(B_{[0,n]})_x$ is isomorphic to $M_{n,x}$, that the map $\phi^n_x$ is given by $\phi^n_x(a)=a\otimes I$ with the unit matrix $I$ of size $d(\theta^{-n-1}(x))$, and that the map $\rho_x^n$ is given by the identity. 

For $n=0$ we have $B_0=B_{[0,0]}=C_0(X)$ and $\rho^0=\id_{C_0(X)}$. This implies $(B_{[0,0]})_x=\C$ $\rho^0_x=\id_\C$ for all $x\in X$. 

Fix $n\in\N$ and $x\in D_{n+1}$. Since $(B_{n+1})_x$ is isomorphic to $M_{n+1,x}$, using Lemma \ref{lem:calculating ideal for core} we obtain that $(B_n')_x$ is isomorphic to $M_{n,x}$. Lemma \ref{lem:inclusion map cores compacts} shows that the inclusion map $\iota^n_x$ of $(B_n')_x$ into $(B_{n+1})_x$ is given by $\iota^n_x(a)=a\otimes I$.

Now comes the induction step. Assume we know that $(B_{[0,n]})_x$ is isomorphic to $M_{n,x}$, and that $\rho^n_x$ is the identity. Using \cite[Proposition 4.5]{pedersen:1999} it is easy to see that the diagram 
\[\begin{tikzcd}
	{M_{n,x}} & {M_{n,x}} \\
	{M_{n+1,x}} & {M_{n+1,x}}
	\arrow["{-\otimes I}", from=1-2, to=2-2]
	\arrow["{-\otimes I}"', from=1-1, to=2-1]
	\arrow["{\mathrm{id}}", from=1-1, to=1-2]
	\arrow["{\mathrm{id}}"', from=2-1, to=2-2]
\end{tikzcd}\]
is a pushout. From (\ref{eq:diagram pushout fibers}) and the uniquess of pushouts we obtain that $(B_{[0,n+1]})_x$ is isomorphic to $M_{n+1,x}$, and that $\phi_x^n$ corresponds to the map $-\otimes I$. We also obtain that the inclusion map of $(B_{n+1})_x$ into $(B_{[0,n+1]})_x$ is given by the identity. Since it is clear from Lemma \ref{lem:calculating ideal for core} that the inclusion map of $(B_{n+1}')_x$ into $(B_{n+1})_x$ is the identity, this implies that $\rho^{n+1}_x$ is the identity as well. 

We are now ready to finish the proof. First note that by Remark \ref{rem:inductive system of C(X)-algebras} the fiber $\CP(\E)^0_x$ of the fixed point algebra at a point $x\in X$ is the limit of the inductive system $((B_{[0,n]})_x,\phi^n_x)$. If $x$ lies in $D_n$ for all $n\in\N$ then we have shown that $(B_{[0,n]})_x$ is isomorphic to $M_{n,x}$, and that $\phi^n_x$ is given by $-\otimes I$. 

If there is $n\in\N_0$ such that $x$ lies in $D_n$ but not in $D_{n+1}$, then we obtain that $(B_{[0,k]})_x$ is isomorphic to $M_{k,x}$ and $\phi^{k-1}_x=-\otimes I$ for all $k\leq n$. The fiber $(B_{n+1})_x$ is isomorphic to the trivial $C^*$-algebra $0$. Using \cite[Proposition 4.5]{pedersen:1999} once more, we have that the diagram 
\[\begin{tikzcd}
	0 & {M_{n,x}} \\
	0 & {M_{n,x}}
	\arrow["{\mathrm{id}}", from=1-2, to=2-2]
	\arrow[from=1-1, to=2-1]
	\arrow[from=1-1, to=1-2]
	\arrow[from=2-1, to=2-2]
\end{tikzcd}\]
is a pushout. This implies $(B_{[0,k]})_x=0$ for all $k>n$, and hence $\CP(\E)^0_x$ is isomorphic to $M_{n,x}$. This includes the case $n=0$, so we are done with computing the fibers.

Since the fibers vanish outside of $D_n$ it suffices to show continuity of $\CP(\E)^0$ as a $C_0(D_n)$-algebra. The fibers over points in $D_n$ are simple, and so the statement follows from Lemma \ref{lem:continuous field of Cstar algebras}.
\end{proof}
\begin{corollary}\label{cor:fixed point algebra for constant rank}
    If the rank of the vector bundle is constant and equal to $d\in\N$, then the fiber over a point $x\in\bigcap_{k\geq0}D_k$ is given by the UHF algebra $M_{d^\infty}$.
\end{corollary}
\begin{corollary}
    Assume that $X$ is compact and has finite covering dimension, $\theta:X\to X$ has domains $U=V=X$ and that the vector bundle has constant rank $d$. Then the fixed point algebra is trivial as a $C(X)$-algebra, meaning that $\CP(\E)^0$ is isomorphic to $ C(X)\otimes M_{d^\infty}$ as $C(X)$-algebras.
\end{corollary}
\begin{proof}
    The statement follows from Theorem \ref{thm: fixed point algebra}, Corollary \ref{cor:fixed point algebra for constant rank} and the main result of \cite{Dadarlat:2008}.
\end{proof}
\begin{remark}
    In the setting of Example \ref{ex:continuous field of cuntz algebras}, Theorem \ref{thm: fixed point algebra} was proven in \cite[Corollary 3]{Vasselli:2005}.
\end{remark}

\section{Ideal structure and the tracial state space}\label{sect:ideal structure and tracial state space}
\subsection{Ideals in Cuntz--Pimsner algebras}\label{subsec: the ideal structure of Cuntz--Pimsner algebras}
In this section we establish some general results about ideals in Cuntz--Pimsner algebras, which we will later use to establish the necessary and sufficient conditions for simplicity of the Cuntz--Pimsner algebra associated to partial automorphisms twisted by vector bundles.

For the remainder of Section~\ref{subsec: the ideal structure of Cuntz--Pimsner algebras}, we set $\E$ to be be a $C^*$-correspondence over a $C^*$-algebra $A$ with structure map $\varphi_\E:A\to \mathcal{L}(\E)$.

\begin{definition}[{\cite[Definitions 4.1,4.8]{Katsura2007}}]
    Let $I$ be an ideal in $A$. We define an idea $\E(I)$ of $A$ by
    \[\E(I)\coloneqq\{\langle\eta,\varphi_\E(a)\xi\rangle_\E:a\in I,\xi,\eta\in\E\},\]
    and say $I$ is \emph{positively invariant} if it contains $\E(I)$. 
\end{definition}

Let $I$ be a positively invariant ideal of $A$. Section \ref{sect:Hilbert modules} shows that $\E I$ is a Hilbert $I$-module and that $\E/\E I$ is a Hilbert $A/I$-module in the natural way. As in Section \ref{sect:Hilbert modules} we write $A_I$ for $A/I$ and $\E_I$ for $\E/\E I$. 

By \cite[Proposition 1.3]{Katsura2007} being positively invariant is equivalent to $\varphi_\E(I)\E$ being contained in $\E I$. Hence the structure map $\varphi_\E$ restricts to $\E I$. This turns $\E I$ into a $C^*$-correspondence over $I$. On the quotient module $(\E)_I$ we can define a structure map by
\[
\varphi_{\E_I}([a]_I)\coloneqq [\varphi_\E(a)]_I.
\]
This is well defined because the positive invariance of $I$ implies that if $a$ lies in $I$, then $[\varphi_\E(a)]=0$. Hence $(\E)_I$ is a $C^*$-correspondence over $A_I$. 

Even if the original $C^*$-correspondence $\E$ is not a bimodule, it is always possible to pass to a bimodule over a different $C^*$-algebra whose Cuntz--Pimsner algebra is $\CP(\E)$. The new $C^*$-algebra is the fixed point algebra, while the bimodule is the first spectral subspace $\CP(\E)^1$. For notational convenience we write $A_b$ instead of $\CP(\E)^0$, and $\E_b$ instead of $\CP(\E)^1$. It follows from the definition of the spectral subspaces that $\E_b$ is indeed a bimodule over $A_b$, with the multiplication in the Cuntz--Pimsner algebra as left and right action and inner products
\begin{align*}
    \langle\xi,\eta\rangle_{\E_b}=\xi^*\eta,\quad\text{and}\quad\prescript{}{\E_b}{\langle}\xi,\eta\rangle=\xi\eta^*
\end{align*}
for $\xi,\eta\in\E_b$.
\begin{proposition}[{\cite[Proposition 10.8]{Katsura2007}}]\label{prop:algebra of correspondence and bimodul}
    The natural embedding of $A_b$ and $\E_b$ into $\CP(\E)$ gives an isomorphism $\CP(\E_b)\cong\CP(\E)$.
\end{proposition}
\begin{definition}
    We say that an ideal $P$ of $\CP(\E)$ is \emph{gauge invariant} if it is invariant under the gauge action, that is, $\gamma_z(P)$ is contained in $P$ for all $z\in\mathbb{T}$. 
\end{definition}
\begin{lemma}\label{lem: ideal cstar correspondence}
    Let $\E$ be a $C^*$-correspondence over a $C^*$-algebra $A$, and let $P$ be a gauge-invariant ideal of $\CP(\E)$. Define an ideal $I_P$ of $A_b$ by $I_P=A_b\cap P$. Then we have isomorphisms $P\cong\CP(\E_bI_P)$ and $\CP(\E)/P\cong \CP((\E_b)_{I_P})$.
\end{lemma}
\begin{proof}
    Combine Proposition \ref{prop:algebra of correspondence and bimodul} with \cite[Theorem 10.6]{Katsura2007}.
\end{proof}

\begin{lemma}\label{lem:quotient bimodule}
   Let $I$ be an ideal of $A$ such that $\varphi_\E(I)\E=\E I$. Then we have $A_bI\cong (AI)_b$ and $A_b/A_bI\cong (A/AI)_b$ as $C^*$-algebras, and $\E_bI\cong(\E I)_b$ and $(\E_b)_I\cong (\E_I)_b$ as $C^*$-correspondences. 
\end{lemma}
\begin{proof}
    According to Proposition 9.3 in \cite{Katsura2007}, the Cuntz--Pimsner algebra $\CP(\E I)$ is a subalgebra of $\CP(\E)$ in the natural way. Hence we can regard $(AI)_b$ as a subalgebra of $\CP(\E)$ as well. As such it is equal to $A_b I$, and the same holds for $(\E I)_b$ and $\E_bI$. Since in both cases left and right multiplication are defined by multiplication in $\CP(\E)$, the identity map on $\CP(\E)$ defines a bijective morphism of $C^*$-correspondences between $(\E I)_b$ and $\E_bI$. From Lemma \ref{lem:bijective implies covariant} we obtain that this morphism is covariant.

    By definition, $A_I=A/A I$ and $\E_I=\E/\E I$ consist of equivalence classes of the form $[a]=\{a+b:b\in I\}$ and $[\xi]=\{\xi+\eta:\eta\in\E I\}$ for $a$ in $A$ and $\xi$ in $\E$. We define a map
    \begin{align*}
        \Pi:(A_I)_b \to (A_b)_I, 
        \qquad [\xi_1]...[\xi_n][\eta_n]^*...[\eta_1]^* \mapsto [\xi_1...\xi_n\eta_n^*...\eta_1^*]\quad\text{for any }n\in\N,
    \end{align*}
    which after extending linearily is a well-defined *-isomorphism. In the same way we can define a bijective linear map $T$ from $(\E_I)_b$ to $(\E_b)_I$. Both $(A_b)_I$ and $(\E_b)_I$ can be regarded as subsets of $\CP(\E)/\CP(\E)I$, and from the definition of multiplication in the quotient algebra and the definitions of $\Pi$ and $T$ it follows that $(\Pi,T)$ is a morphism of $C^*$-correspondences. Applying Lemma \ref{lem:bijective implies covariant} again yields that this morphism is covariant.
 \end{proof}
 
\begin{proposition}\label{prop:ideal cstar correspondence}
   Let $\E$ be a $C^*$-correspondence over a $C^*$-algebra $A$, and let $I$ be an ideal of $A$ such that $\varphi_\E(I)\E=\E I$. Then $\CP(\E)I$ is a gauge-invariant ideal in $\CP(\E)$, and we have $\CP(\E)I\cong\CP(\E I)$ as well as $\CP(\E)/(\CP(\E)I)\cong \CP(\E_I)$. In other words the exact sequence 
   \begin{align*}
       0\to I\to A\to A/I\to 0
   \end{align*}
    gives rise to an exact sequence
    \begin{align*}
        0\to\CP(\E I)\to\CP(\E)\to\CP(\E_I)\to 0.
    \end{align*}
\end{proposition}
\begin{proof}
    First note that if $\varphi_\E(I)\E$ equals $\E I$ then $P=\CP(\E)I=I\CP(\E)$ is indeed an ideal in $\CP(\E)$. Since $\gamma_z(\CP(\E)I)=\gamma_z(\CP(\E))\gamma_z(I)=\CP(\E)I$, it is gauge invariant. 
    
    We have $I_P=A_b\cap \CP(\E)I=A_bI$ and hence $\E_bI_P=\E_bA_bI=\E_bI$. This implies $A_b/(A_bI_P)=A_b/(A_bI)$ as well as $(\E_b)_{I_P}=\E_b/(\E_bI_P)=\E_b/(\E_bI)=(\E_b)_I$. Now combine Proposition \ref{prop:algebra of correspondence and bimodul}, Lemma \ref{lem: ideal cstar correspondence} and Lemma \ref{lem:quotient bimodule}.
\end{proof}

\subsection{Restricting the \texorpdfstring{$C^*$}{C*}-correspondence}\label{sect: restriction}
This section compiles a number of results which are related to restricting the partial automorphism and hence the $C^*$-correspondence in different ways, and which we will need in later chapters. 
We write $\E$ for the $C^*$-correspondence $\Gamma_0(\V,\theta)$ over $C_0(X)$, where $\theta:U\to V$ is a partial automorphism on a locally compact second countable Hausdorff space $X$, and $\V$ is a vector bundle over $U$. 

\begin{definition}\label{def:restriction partial auto}
    For a partial automorphism $\theta:U\to V$ and an open subset $W\subset U$ we define a partial automorphism $\theta^{(W)}:W\to\theta(W)$ on $X$ by $\theta^{(W)}(x)=x$ for all $x\in W$. The $n$-th domain of $\theta^{(W)}$ is denoted by $D_n^{(W)}$. We write $\E^{(W)}$ for the $C_0(X)$-correspondence $\Gamma_0(\V\vert_W,\theta^{(W)})$. 
\end{definition}

\begin{lemma}\label{lem:restriction of partial auto subcorrespondence}
The inclusion map from $\E^{(W)}$ to $\E$ together with the identity on $C_0(X)$ is an injective covariant morphism of $C^*$-correspondences, and thus $\E^{(W)}$ is a $C^*$-subcorrespondence of $\E$. Hence $\CP(\E^{(W)})$ is a subalgebra of $\CP(\E)$. 
\end{lemma}

\begin{proof}
It is easy to check that the inclusion map is a morphism of $C^*$-correspondences, so what is left to prove is covariance. 

What we have to show is that $J_{\E^{(W)}}$ is contained in Katsura's ideal $J_{\E}$, and that the equality $\varphi_\E(f)=\Psi_{T}(\varphi_{\E^{(W)}}(f))$ holds for all $f\in J_{\E^{(W)}}$, where $T:\E^{(W)}\to\E$ is the inclusion map. Since $J_{\E^{(W)}}=C_0(\theta(W))$ and $J_{\E}=C_0(V)$ by Lemma \ref{lem:katsura ideal}, the first part is true. 

To see the equality, take $f\in J_{\E^{(W)}}$. We can approximate $f$ up to arbitrary precision by a function $f_0$ whose support is contained in some compact set $K\subset \theta(W)$. By Lemma \ref{lem:frame for submodule} there exist sections $\xi_1,...,\xi_N$ in $\E^{(W)}$ for some $N\in\N$ such that 
\[\sum_{n=1}^N\theta_{\xi_n,\xi_n}\eta=\eta,\]
holds for all $\eta\in \E  C_0(\interior\theta^{-1}(K))$. This implies that the fiber of $\sum_{n=1}^N\theta_{\xi_n,\xi_n}$ over every point $x\in \interior K$ is the identity. Note that when we speak of the fiber over a point we are referring to the $C_0(X)$-structure induced by the structure map, and not by right multiplication. Since $f_0$ vanishes outside of $K$ this implies 
\[\varphi_{\E^{(W)}}(f_0)=\varphi_{\E^{(W)}}(f_0)\sum_{n=1}^N\theta_{\xi_n,\xi_n}=\sum_{n=1}^N\theta_{\varphi_{\E^{(W)}}(f_0)\xi_n,\xi_n}.\]
 We obtain
\[\psi_T(\varphi_{\E^{(W)}}(f_0))=\psi_T\left(\sum_{n=1}^N\theta_{\varphi_{\E^{(W)}}(f_0)\xi_n,\xi_n}\right)=\sum_{n=1}^N\theta_{\varphi_{\E}(f_0)\xi_n,\xi_n}=\varphi_\E(f_0),\]
which completes the proof.
\end{proof}
\begin{proposition}\label{prop: inductive limit}
Let $X$ be a locally compact Hausdorff space, and let $\theta:U\to V$ be a partial automorphism on $X$. Take an increasing sequence $(U_k)_{k=1}^\infty$ of open subsets in $X$ such that $U=\bigcup_{k=1}^\infty U_k$. Then
\begin{align*}
    \CP(\E)=\lim_{\to}\CP(\E^{(U_k)}).
\end{align*}
\end{proposition}
\begin{proof}
    As in the proof of Lemma \ref{lem:inductive limit exhaustion by compact sets}, the union of all $\E^{(U_k)}$ is dense in $\E$. The statement then follows from the Lemmas \ref{lem:inductive limit correspondences} and \ref{lem:restriction of partial auto subcorrespondence}.
\end{proof}

A subset of a topological space is said to be \emph{locally closed} if it can be written as an intersection of a closed and an open set. In particular, all open sets are locally closed.
\begin{definition}\label{def:restriction space}
Let $W$ be a locally closed, $\theta$-invariant subset of $X$. Denote by $\theta_W:U\cap W\to V\cap W$ the partial automorphism on $W$ obtained by restricting $\theta$. We write $\E_W$ for the $C_0(W)$-correspondence $\Gamma_0(\V|_{W},\theta_W)$. 
\end{definition}

\begin{proposition}\label{prop: short exact sequence}
Let $W$ be an open $\theta$-invariant subset of $X$. Then $ \CP(\E)C_0(W)$ is an ideal of $\CP(\E)$. We have
\begin{align*}
     \CP(\E)C_0(W)\cong\CP(\E_W)\quad\text{ and }\quad \CP(\E)/\CP(\E)C_0(W)\cong \CP(\E_{X\backslash W}).
\end{align*}
Hence
\begin{align*}
    0\to \CP(\E_W)\to\CP(\E)\to\CP(\E_{X\backslash W})\to 0
\end{align*}
is a short exact sequence of $C^*$-algebras. 
\end{proposition}
\begin{proof}
First note that if $W$ is an open, $\theta$-invariant subset of $X$, then $C_0(W)$ is an ideal of $C_0(X)$ such that $\varphi_\E(C_0(W))\E=\E C_0(W)$. Hence we are in the setting of Proposition \ref{prop:ideal cstar correspondence}. Note that $\E C_0(W)$ is isomorphic to $\E_W$ as $C^*$-correspondences over $C_0(W)$, and that $\E/\E C_0(W)$ is isomorphic to $\E_{X\backslash W}$ as $C^*$-correspondences over $C_0(X\backslash W)$. The assertions follow.  
\end{proof}

\subsection{Ideal structure}

Before describing the ideal structure of our Cuntz--Pimser algebras, we need a technical lemma about partial actions. It is a generalization of \cite[Lemma 11.18]{Giordano2018}, and the only result in this paper proved for partial actions of general discrete groups. The proof is only a minor variation of the proof in \cite{Giordano2018}, one just needs to insert the domains $D_g$ at the approprate places. We thus omit it.
\begin{lemma}\label{lem:functions vanishing}
Let $X$ be a locally compact space, $G$ a discrete group, and $\{\theta_g,D_g\}_{g\in G}$ a free partial action on $X$. For any finite set $F$ contained in $G\backslash\{ e\}$, any compact set $C$ contained in $X$, and any collection of compact sets $\{K_g\}_{g\in F}$ with $K_g$ contained in $D_g$, there exists $n\in\N$ and elements $h_j\in C_0(X)$ for $1\leq j\leq n$ such that $|h_j(x)|=1$ for all $x\in C$ and
\begin{align*}
    \sum_{j=1}^nh_j(x)\overline{h_j(\theta_{g^{-1}}(x))}=0
\end{align*}
for all $g\in F$ and $x\in K_g$.
\end{lemma}

We now return to the case of partial $\mathbb{Z}$-actions. For the remainder of the section we fix a partial automorphism $\theta: U \to V$ of $X$ and a vector bundles $\mathcal{V}$ over $X$. As in the previous sections we set $\E=\Gamma_0(\V,\theta)$. From now on we will completely omit the universal covariant representation, and simply write $f$ for $\pi_\E(f)$ and $\xi$ for $t_\E(\xi)$, for $f \in C(X)$ and $\xi \in \E$.

\begin{lemma}\label{lem:conditional expectation approximation}
Assume that $\theta$ is free. Then for every $\varepsilon>0$ and every $a\in\CP(\E)$ there exist $m\in\N$ and functions $h_1,...,h_m\in C_0(X)$ such that 
    \begin{align*}
        \norm{\frac{1}{m}\sum_{j=1}^m\overline{h_j}ah_j-\Phi(a)}<\varepsilon,
    \end{align*}
    where $\Phi$ is the conditional expectation onto $\CP(\E)^0$.
\end{lemma}

\begin{proof}
    Take $\varepsilon>0$ and $a\in\CP(\E)$. According to Proposition \ref{prop: inductive limit}, there exists an open set $W$ whose closure is compact and contained in $U$, and such that $\norm{a-b}<\varepsilon/2$ for some $b$ in $\CP(\E^{(W)})$. The algebraic direct sum over all spectral subspaces $\CP(\E^{(W)})^n$ is dense in $\CP(\E^{(W)})$. Therefore, by Proposition \ref{prop:tensor product fullness}, we can assume that $b$ has the form $\sum_{n=-N}^N g_n b_n$ for $g_n\in C_0(D^{(W)}_{n})$, $b_n\in\CP(\E^{(W)})^n$, and some $N\in\N$. We can further assume that the fibers of $g_0 b_0$ in the fixed point algebra (see Theorem \ref{thm: fixed point algebra}) vanish outside of some compact set $C$ contained in $U$.

     Because the closure of $W$ is compact and contained in $U$, one can show using the formulae $D^{(W)}_n=\theta(D^{(W)}_{n-1}\cap D^{(W)}_{-1})$ and $D^{(W)}_{-n}=\theta(D^{(W)}_{1-n}\cap D^{(W)}_{1})$ that the closure of $D^{(W)}_n$ is compact and contained in $D_n$, for every $n\in\Z$. Define $K_n$ to be the closure of $D^{(W)}_n$. Using Lemma \ref{lem:functions vanishing} we obtain $m\in N$ and functions $h_j\in C_0(X)$, $1\leq j\leq m$, such that $|h_j(x)|=1$ for all $x\in C$ and for $-N\leq n\leq N$, $n\neq 0$, we have
    \begin{align}\label{eq:lem:conditional expectation approximation 2}
        \sum_{j=1}^m\overline{h_j(\theta^n(x))}h_j(x)=0, \quad \text{for all }x\in K_{-n}.
    \end{align}
     
    Now we calculate
   \begin{align*}
       \frac{1}{m}\sum_{j=1}^m\overline{h}_jbh_j&=\frac{1}{m}\sum_{j=1}^m\sum_{n=-N}^N\overline{h}_jg_n b_nh_j\\&=\frac{1}{m}\sum_{j=1}^m\sum_{n=-N}^N b_n(g_n\overline{h}_j\circ\theta^n)h_j=(\star).
   \end{align*}
   Each $g_n$ vanishes outside of $D^{(W)}_{n}$, and hence each $g_n\overline{h}_j\circ\theta^n$ vanishes outside of $D^{(W)}_{-n}$ which is contained in $K_{-n}$. Equation (\ref{eq:lem:conditional expectation approximation 2}) yields
   \begin{align*}
       \sum_{j=1}^m(g_n\overline{h}_j\circ\theta^n)h_j=0
   \end{align*}
   for all $n\neq0$. Hence
   \begin{align*}
       (\star)&=\frac{1}{m}\sum_{j=1}^m b_0g_0\overline{h}_jh_j=\frac{1}{m}\sum_{j=1}^m b_0g_0\\&= b_0g_0=\Phi(b).
   \end{align*}
   We obtain
   \begin{align*}
        \norm{\frac{1}{m}\sum_{j=1}^m\overline{h_j}ah_j-\Phi(a)}&\leq \norm{\frac{1}{m}\sum_{j=1}^m\overline{h_j}ah_j-\frac{1}{m}\sum_{j=1}^m\overline{h_j}bh_j}+\norm{\frac{1}{m}\sum_{j=1}^m\overline{h_j}bh_j-\Phi(a)}\\&=\norm{\frac{1}{m}\sum_{j=1}^m\overline{h_j}(a-b)h_j}+\norm{\Phi(b)-\Phi(a)}\\&\leq \frac{1}{m}\sum_{j=1}^m\norm{\overline{h_j}}\norm{a-b}\norm{h_j}+\norm{\Phi}\norm{b-a}\\&= \norm{a-b}+\norm{b-a}<\varepsilon,
   \end{align*}
   as required.
\end{proof}

\begin{proposition}\label{prop:ideal intersection with fixed point algebra}
Assume that the partial automorphism is free. If $J$ is a nonzero ideal in $\CP(\E)$, then $J$ has non-trivial intersection with the fixed point algebra, that is,
\begin{align*}
\CP(\E)^0\cap J\neq\{0\}.
\end{align*}
In fact,
\begin{align*}
    C_0(X)\cap J\neq\{0\}.
\end{align*}
\end{proposition}

\begin{proof}
Let $J$ be a nonzero ideal in $\CP(\E)$ and $a\in J$. Choose $\varepsilon>0$. According to Lemma \ref{lem:conditional expectation approximation} there exist $m\in\N$ and functions $h_1,...,h_m\in C_0(X)$ such that $\frac{1}{m}\sum_{j=1}^m\overline{h_j}ah_j$ approximates $\Phi(a)$ up to $\varepsilon$. Clearly $\frac{1}{m}\sum_{j=1}^m\overline{h_j}ah_j$ lies in $J$. Since $\varepsilon$ was arbitrary and $J$ is closed, this implies that $\Phi(a)$ lies in $J$. 

Let $a$ in $J$ be nonzero. Then $a^*a$ is nonzero, positive and lies in $J$. We have $\Phi(a^*a)\neq0$, because $\Phi$ is faithful. This proves that the intersection of $J$ with the fixed point algebra is nontrivial. 

To show that $C_0(X) \cap J \neq \{0\}$, recall that the fixed point algebra is a continuous $C_0(X)$-algebra by Theorem \ref{thm: fixed point algebra}. Hence by \cite[Lemma 1.8]{Fell:1961}, every ideal $I$ in $\CP(\E)^0$ is of the form
\begin{align}\label{eq:ideal C(X) algebra}
    I=\{a\in\CP(\E)^0:a_x\in I_x\;\text{for all }x\in X\}
\end{align}
where $I_x=\{a(x):a\in I\}$. It is not hard to see that $I_x$ is an ideal in $\CP(\E)^0_x$ for all $x\in X$. Now write $I\coloneqq \CP(\E)^0\cap J$. Clearly $I$ is an ideal in the fixed point algebra. Since the fibers of the fixed point algebra are simple, we either have $I_x=0$ or $I_x=\CP(\E)^0_x$. There must be an open set in $X$ such that the fibers of $I$ over the set are all nonzero, because if the set of points such that $I_x=0$ were dense, then $I$ would be zero by (\ref{eq:ideal C(X) algebra}) and continuity of the fixed point algebra. Again by (\ref{eq:ideal C(X) algebra}), and because all the fibers of the fixed point algebra are unital, this implies that there is a nonzero continuous function contained in $I$, and hence in $J$.
\end{proof}
\begin{proposition}\label{prop:invariant}
Let $J$ be an ideal of $\CP(\E)$. There exists a $\theta$-invariant open subset $W$ of $X$ such that $C_0(W)=J\cap C_0(X)$.
\end{proposition}
\begin{proof}
Write $I:=J\cap C_0(X)$. We will show that $\{f\circ\theta:f\in I\cap C_0(V)\}$ and $\{f\circ\theta^{-1}:f\in I\cap C_0(U)\}$ are both contained in $I$.

Let $\{v_n\}_{n\in\N}$ be an approximate identity for $C_0(U)$. Take $f$ in $I\cap C_0(V)$ and $\varepsilon>0$. Then $f\circ\theta$ lies in $C_0(U)$. Thus there exists $n_0\in\N$ such that $v_n(f\circ\theta)$ and $f\circ\theta$ are $\varepsilon$-close for all $n\geq n_0$. 

By Remark \ref{rem:sigma fp module is full} we have $\langle\E,\E\rangle_\E=C_0(U)$. Hence for any $n\geq n_0$ there exist $k_n\in\N$ and elements $\xi_1^n,\eta_1^n,...,\xi_{k_n}^n,\eta_{k_n}^n\in\E$ such that the linear combination $\sum_{i=1}^{k_n}\xi_i^*\eta_i$ is $\varepsilon/\norm{f}$-close to $v_n$. Then for every $n\geq n_0$ we have
\[
I\ni \sum_{i=1}^{k_n}(\xi_i^n)^*f\eta_i^n=\sum_{i=1}^{k_n}(\xi_i^n)^*\eta_i^n(f\circ\theta)\approx_\varepsilon v_n(f\circ\theta)\approx_\varepsilon (f\circ\theta).
\]
Since $\varepsilon$ was arbitrary and $I$ is closed, this shows that $f\circ\theta$ lies in $I$.

If $f$ lies in $I\cap C_0(U)$ then essentially the same argument, using that $C_0(V)=\varphi_\E^{-1}(\K(\E))$ holds by Lemma \ref{lem:katsura ideal}, shows that $f\circ\theta^{-1}$ lies in $I$.

To finish the proof, note that $I$ is an ideal of $C_0(X)$. Hence there exists an open subset $W$ of $X$ such that $I=C_0(W)$. It follows from the first part of the proof that $W$ is $\theta$-invariant.
\end{proof}

We are now ready to prove the main result of this section. It is a straightforward adaptation of \cite[Theorem 29.9]{Exel:2017}, the only caveat being that we require freeness instead of the weaker notion of topological freeness. 
\begin{theorem}\label{thm:ideals in Cuntz--Pimsner algebra}
    Assume that the partial automorphism $\theta$ is free. Then there is a bijective correspondence between $\theta$-invariant open subsets of $X$ and ideals in $\CP(\E)$, given by mapping an invariant open set $W$ to $\CP(\E)C_0(W)$. 
\end{theorem}
\begin{proof}
    That $\CP(\E)C_0(W)$ is indeed an ideal of $\CP(\E)$ whenever $W$ is a $\theta$-invariant open set follows from Proposition \ref{prop: short exact sequence}. Since the intersection of $\CP(\E)C_0(W)$ with $C_0(X)$ is $C_0(W)$, the correspondence is injective. It is left to show that it is surjective.

    Let $J$ be an ideal of $\CP(\E)$. By Propositions \ref{prop:ideal intersection with fixed point algebra} and \ref{prop:invariant} there exists a non-empty, $\theta$-invariant open set $W$ such that $J\cap C_0(X)=C_0(W)$. Since $J$ contains $C_0(W)$ and is an ideal, we obtain that $\CP(\E)C_0(W)$ is contained in $J$. To see that $J \subset \CP(\E)C_0(W)$, consider the short exact sequence 
\[\begin{tikzcd}
	0 & {\mathcal{O}(\mathcal{E})C_0(W)} & {\mathcal{O}(\mathcal{E})} & {\mathcal{O}(\mathcal{E}_{X\backslash W})} & 0
	\arrow[from=1-1, to=1-2]
	\arrow[from=1-2, to=1-3]
	\arrow[from=1-4, to=1-5]
	\arrow["q", from=1-3, to=1-4]
\end{tikzcd}\]
from Proposition \ref{prop: short exact sequence}. We claim that $q(J)$ has trivial intersection with $C_0(X\backslash W)$. By Proposition \ref{prop:ideal intersection with fixed point algebra} this would imply that $q(J)$ is trivial, and hence $J$ is contained in $\ker q=\CP(\E) C_0(W)$. 

Take $a$ from $q(J)\cap C_0(X\backslash W)$. There exists $b\in J$ such that $q(b)=a$. Since $q(C_0(X))=C_0(X\backslash W)$ there also exists $f\in C_0(X)$ such that $q(f)=a$. Then $b-f$ lies in $\CP(\E)C_0(W)$. Since $\CP(\E)C_0(W)$ is contained in $J$ and $b$ lies in $J$ we obtain that $f$ lies in $J\cap C_0(X)=C_0(W)$. Hence $a=q(f)=0$, and we are done.
\end{proof}
\begin{corollary}\label{cor:simplicity}
    If the partial automorphism $\theta$ is free, then the Cuntz--Pimsner algebra $\CP(\E)$ is simple if and only if $\theta$ is minimal. 
\end{corollary}

\subsection{The tracial state space}
We now investigate the tracial state space of our Cuntz--Pimsner algebra, which we denote by $T(\CP(\E))$. It is of great interest since it is an important part of the Elliott invariant. 

The set-up is exactly the same as in the previous sections. The statements as well as their proofs are very similar to the ones in Section 4 of \cite{adamo2023}. 

\begin{lemma}\label{lem:changing order}
Suppose that $\V$ is a line bundle. Then, for every $j>0$ and for every $\xi_j, \eta_j \in \CP(\E)^{j}\cong\E^{\otimes j}$ we have 
\begin{align*}
   \xi_j^*\eta_j\in C_0(D_{-j}),\;\eta_j\xi_j^*\in C_0(D_j)\quad\text{and}\quad\xi_j^*\eta_j=\eta_j\xi_j^*\circ\theta^{j}. 
\end{align*}  
\end{lemma}
\begin{proof}
 Note that $\E$ is a $C_0(X)$-bimodule by Proposition \ref{prop:line bundle bimodule}. Then that $\xi_j^*\eta_j$ and $\eta_j\xi_j^*$ lie in $C_0(D_{-j})$ and $C_0(D_j)$, respectively, follows from Proposition \ref{prop:tensor product fullness}. We show that  $\eta_j\xi_j^* \in C_0(D_j)$ inductively. The proof that  $\xi_j^*\eta_j \in C_0(D_{-j})$ is completely analogous. By Proposition \ref{prop:line bundle bimodule} we have
\begin{align*}
\xi^*\eta=\langle\xi,\eta\rangle_\E=\prescript{}{\E}{\langle}\eta,\xi\rangle\circ\theta=\eta\xi^*\circ\theta.
\end{align*}
for all $\eta$ and $\xi$ in $\E$. This shows the claim for $j=1$. Now assume$j < 1$ and that the claim holds for $j-1$. Let $\xi$ and $\eta$ be in $\E$ and $\xi_{j-1}$ and $\eta_{j-1}$ be in $\E^{\otimes(j-1)}$. Then
\begin{align*}
(\xi_{j-1}\xi)^*\eta_{j-1}\eta&=\xi^*\xi_{j-1}^*\eta_{j-1}\eta=\xi^*(\xi_{j-1}^*\eta_{j-1}\eta)\\&=\xi_{j-1}^*\eta_{j-1}\eta\xi^*\circ\theta.
\end{align*}
On the other hand,
\begin{align*}
\eta_{j-1}\eta(\xi_{j-1}\xi)^*=\eta_{j-1}\eta\xi^*\xi_{j-1}^*=(\eta_{j-1}\eta\xi^*)\xi_{j-1}^*=\xi_{j-1}^*\eta_{j-1}\eta\xi^*\circ\theta^{-j+1}.
\end{align*}
 The result follows. 
\end{proof}
For free actions on a compact Hausdorff space $X$, tracial states of the crossed product are in bijective correspondence with measures on $X$ that are invariant under the action. We will see that a similar correspondence holds for Cuntz--Pimsner algebras associated to partial automorphisms twisted by vector bundles. We first need to define the appropriate notion of invariant measure. The reader may wish to compare this definition to that of \cite{Scarparo:2017}, which is stated for more general groups.

\begin{definition}\label{def:invariant measure}
A Radon probability measure $\mu$ on $X$ is called \emph{$\theta$-invariant} if 
\begin{align*}
\mu(\theta^{-1}(Y))=\mu(Y)
\end{align*}x
for all measurable sets $Y\subset V$. The set of all such measures is denoted by $M^1_\theta(X)$. It is a convex subset of the dual space $C_0(X)'$, and it is compact if endowed with the weak*-topology. 
\end{definition}
\begin{remark}\label{rem:invariance of measures and states}
Let $\mu$ be a Radon probability measure on $X$. Let $\tau_\mu$ be the corresponding state on $C_0(X)$, namely 
\begin{align*}
    \tau_\mu(f)=\int_Xfd\mu,
\end{align*} for any $f\in C_0(X)$. Then $\mu$ lies in $M^1_\theta(X)$ if and only if
\begin{align*}
\tau_\mu(f\circ\theta)=\tau_\mu(f)
\end{align*}
for all $f\in C_0(V)$. This is equivalent to $\mu(\theta(Y))=\mu(Y)$ for all measurable sets $Y\subset U$, and to $\tau_\mu(f\circ\theta^{-1})=\tau_\mu(f)$ for all $f\in C_0(U)$.
\end{remark}
We have the following analogue of the classical theorem of Krylov--Bogoliubov: 
\begin{theorem}\label{thm:krylov bogoliubov}
    If $X$ is compact and Hausdorff, then for every point $x\in X$ there exists an invariant probability measure whose support is contained in $\overline{\orb(x)}$. In particular $M_\theta^1(X)$ is nonempty.
\end{theorem}
\begin{proof}
Up to a metrizability assumption which can be removed, this is a special case of \cite[Theorem 3.15]{arbieto:2009}. The existence also follows from \cite[Proposition 2.7]{Scarparo:2017} together with the fact that $\Z$ as an abelian group is supramenable. 
\end{proof}
\begin{proposition}\label{prop:trace space}
 If $X$ is compact and $\V$ is a line bundle then $T(\CP(\E))\neq0$.
\end{proposition}
\begin{proof}
 Let $\mu$ be an $\theta$-invariant probability measure on $X$ in the sense of Definition \ref{def:invariant measure}. Let $\tau_\mu$ be the associated state as in Remark \ref{rem:invariance of measures and states}. We claim that $\tau_\mu\circ\Phi$ is tracial.

Take $a=\sum_{n=-N}^N\xi_n$ and $b=\sum_{n=-N}^N\eta_n$ for some $N\in\N$ and $\xi_n$ and $\eta_n$ in $\E^{\otimes n}$. It is enough to consider such elements according to Section \ref{sec:grading and the fixed point algebra}. Then we have
\begin{align*}
\Phi(ab)&=\Phi\left(\sum_{j,k=-N}^N\xi_j\eta_k\right)\\&=\sum_{j=-N}^N\xi_{-j}\eta_j=\sum_{j=-N}^N\eta_j\xi_{-j}\circ\theta^{j}
\end{align*}
where the last equality holds because of Lemma \ref{lem:changing order}. Since $\mu$ is $\theta$-invariant, this shows that $\tau_\mu\circ\Phi(ab)=\tau_\mu\circ\Phi(ba)$, which proves the claim. 

If the space $X$ is compact, then it has a $\theta$-invariant probability measure by Theorem \ref{thm:krylov bogoliubov}. It follows that $T(\CP(\E))\neq0$.  
\end{proof}
Assume that $X$ is compact. Recall that $M_\theta^1(X)$ denotes the space of $\theta$-invariant probability Radon measures on $X$. 

\begin{proposition}\label{prop:affine homeomorphism}
If $X$ is compact, $\theta$ is free, and $\V$ is a line bundle, then the map from $M^1_\theta(X)$ to $T(\CP(\E))$ sending $\mu$ to $\tau_\mu\circ\Phi$ is an affine homeomorphism.
\end{proposition}

\begin{proof}
From the proof of Proposition \ref{prop:trace space} we know  that the map is well defined, meaning that $\tau_\mu\circ\Phi$ is indeed a trace on $\CP(\E)$ for every $\mu$ in $M^1_\theta(X)$. It is injective because $\Phi$ maps onto $C(X)$, and hence $\tau_\mu\circ\Phi=\tau_\nu\circ\Phi$ implies $\tau_\mu=\tau_\nu$, which can only happen if $\mu$ and $\nu$ coincide. 

To prove surjectivity let $\tau$ be a trace on $\CP(\E)$. There exists a measure $\mu$ such that $\tau|_{C(X)}=\tau_\mu$. Take $a\in\CP(\E)$ and $\varepsilon>0$. According to Lemma \ref{lem:conditional expectation approximation} there exist $m\in\N$ and functions $h_1,...,h_m\in C(X)$ such that $|h_j|=1$ for all $1\leq j\leq m$, and such that $\frac{1}{m}\sum_{j=1}^m\overline{h_j}ah_j$ is $\varepsilon$-close to $\Phi(a)$.
    Hence
    \begin{align*}
    \norm{\tau(a)-\tau|_{C(X)}\circ\Phi(a)}&=\norm{\tau \left (\frac{1}{m}\sum_{j=1}^m\overline{h_j}ah_j \right)-\tau\circ\Phi(a)}=\norm{\tau \left (\frac{1}{m}\sum_{j=1}^m\overline{h_j}ah_j-\Phi(a) \right)}\\&\leq \norm{\frac{1}{m}\sum_{j=1}^m\overline{h_j}ah_j-\Phi(a)}<\varepsilon.
    \end{align*}
    Since $\varepsilon$ was arbitrary this implies
    \begin{align*}
    \tau=\tau|_{C(X)}\circ\Phi=\tau_\mu\circ\Phi.
    \end{align*}
    It is left to show that $\tau|_{C(X)}$ is indeed $\theta$-invariant, meaning $\tau|_{C(X)}(f\circ\theta)=\tau|_{C(X)}(f)$ for all $f$ in $C_0(V)$. Invariance of $\mu$ then follows from Remark \ref{rem:invariance of measures and states}. 
    
    Using an atlas for $\V$ and a suitable partition of unity, we can construct a countable frame $(\xi_i)_{i=1}^\infty$ of $\Gamma_0(\V,\theta)$ similarly to Lemma \ref{lem:frame vector bundle}. It has the property 
     \[\sum_{i=1}^\infty\xi_i^*\xi_i=1.\]
    Under the left action any $f\in C_0(V)$ acts on $\Gamma_0(\V,\theta)$ as a compact operator, which implies 
    \[f\sum_{i=1}^\infty\xi_i\xi_i^*=f.\]
    Thus
    \begin{align*}
    \tau(f)&=\tau\left(f\sum_{i=1}^\infty\xi_i\xi_i^*\right)=\sum_{i=1}^\infty\tau(f\xi_i\xi_i^*)\\&=\sum_{i=1}^\infty\tau(\xi_i(f\circ\theta)\xi_i^*)=\sum_{i=1}^\infty\tau(\xi_i^*\xi_i(f\circ\theta))\\&=\tau(f\circ\theta)
    \end{align*}
    where in the second last step we have used that $\tau$ is a trace. This shows that $\tau|_{C(X)}$ is $\theta$-invariant.
    
    We have thus established that the mapping $\mu\to\tau_\mu\circ\Phi$ is bijective. It is clear that it is affine. Furthermore, by definition, a net $\{\mu_\lambda\}$ of measures in $M^1_\theta(X)$ converges to $\mu\in M^1_\theta(X)$ if and only if $\{\tau_{\mu_\lambda}\}$ converges to $\tau_\mu$ in the weak*-topology of $C(X)'$. This, however, is equivalent to $\{\tau_{\mu_\lambda}\circ\Phi\}$ converging to $\tau\circ\Phi$ in $T(\CP(\E))$, because $\Phi$ maps onto $C(X)$. Hence the mapping $\mu\to\tau_\mu\circ\Phi$ is a homeomorphism. This concludes the proof. 
\end{proof}



\end{document}